\documentclass[review]{siamart1116}

\usepackage{latexsym}
\usepackage{amsmath}
\usepackage{amssymb}
\usepackage{amsfonts}
\usepackage[latin1]{inputenc}
\usepackage{mathrsfs}
\usepackage{amsfonts}
\usepackage{graphicx}
\usepackage{colortbl,dcolumn}
\usepackage{amsmath}
\usepackage{amsfonts}
\usepackage{graphicx}
\usepackage{algorithmic}
\ifpdf
  \DeclareGraphicsExtensions{.eps,.pdf,.png,.jpg}
\else
  \DeclareGraphicsExtensions{.eps}
\fi

\newtheorem{as}{Assumption}[section]
\newtheorem{tm}{Theorem}[section]
\newtheorem{rk}{Remark}[section]

\newtheorem{prop}{Proposition}[section]
\newtheorem{lm}{Lemma}[section]

\newcommand{\E}{\mathbb E}
\newcommand{\PP}{\mathbb P}
\newcommand{\N}{\mathbb N}
\newcommand{\R}{\mathbb R}

\newcommand{\LL}{\mathcal L}
\newcommand{\OO}{\mathcal O}
\newcommand{\HH}{\mathbb H}

\newcommand{\FFF}{\mathscr F}
\newcommand{\<}{\langle}
\renewcommand{\>}{\rangle}

\ifpdf
  \DeclareGraphicsExtensions{.eps,.pdf,.png,.jpg}
\else
  \DeclareGraphicsExtensions{.eps}
\fi

\newcommand{\TheTitle}{Strong and weak convergence rates of a spatial approximation for stochastic partial differential equation with one-sided Lipschitz coefficient}
\newcommand{\TheAuthors}{Jianbo Cui and  Jialin Hong}

\headers{Convergence of FEM for non-sided Lipschitz SPDE}{\TheAuthors}

\title{{\TheTitle}\thanks{Submitted to the editors in DATE.
\funding{This work was supported by National Natural Science Foundation of China (No. 91630312, No. 91530118, No.11021101  and No. 11290142).}}}

\author{
Jianbo Cui 
\thanks{1. LSEC, ICMSEC, 
Academy of Mathematics and Systems Science, Chinese Academy of Sciences, Beijing,  100190, China\qquad
2. School of Mathematical Science, University of Chinese Academy of Sciences, Beijing, 100049, China 
(\email{jianbocui@lsec.cc.ac.cn}(corresponding author), \email{hjl@lsec.cc.ac.cn})}
\and 
Jialin Hong \footnotemark[2]
}

\usepackage{amsopn}

\ifpdf
\hypersetup{
  pdftitle={\TheTitle},
  pdfauthor={\TheAuthors}
}
\fi

\usepackage{ulem}
\usepackage{subfigure}
\usepackage{graphicx}
\usepackage{enumerate}
\usepackage{amsmath,bm}
\usepackage{latexsym}
\usepackage{float}
\usepackage{latexsym}
\usepackage{amsmath}
\usepackage{amssymb}
\usepackage{amsfonts}
\usepackage{verbatim}
\usepackage[latin1]{inputenc}
\usepackage{mathrsfs}
\usepackage{amsfonts}
\usepackage{graphicx}
\usepackage{colortbl,dcolumn}
\usepackage{amsmath}
\usepackage{psfrag}
\usepackage{booktabs}
\usepackage{hyperref}
\usepackage{lipsum}
\usepackage{amsfonts}
\usepackage{graphicx}
\usepackage{epstopdf}
\usepackage{algorithmic}
\allowdisplaybreaks[4]

\begin{document}

\maketitle

\begin{abstract}
Strong and weak approximation errors of a spatial finite element method are analyzed for  stochastic partial 
differential equations(SPDEs) with one-sided Lipschitz coefficients, including the stochastic Allen--Cahn equation, driven by additive noise.
In order to give the strong convergence rate of the  finite element method,
we present an appropriate  decomposition
and some a priori estimates of the discrete stochastic convolution. To the best of our knowledge, there has been no essentially sharp weak convergence rate of spatial approximation for parabolic SPDEs
with non-globally Lipschitz coefficients. 
To investigate the weak error,
we first regularize the original equation by the splitting technique and derive the 
regularity estimates of the corresponding regularized Kolmogorov equation.
Meanwhile, we present the regularity estimate in Malliavin sense and the refined estimate of the  finite element method. 
Combining the 
regularity estimates of regularized Kolmogorov equation  with Malliavin integration by parts formula, the weak convergence rate  is shown to be twice the strong convergence rate. 
\end{abstract}


\begin{AMS}
{60H35}, {60H15, 65M15.}
\end{AMS}

\section{Introduction}
Both strong and  weak convergence rates of  
numerical approximations for SPDEs with globally Lipschitz continuous  and regular nonlinearities  have 
been studied over past decades. 
In contrast to the Lipschitz coefficient case,  strong and weak convergence rates of numerical approximations for 
SPDEs with non-globally Lipschitz continuous nonlinearities, especially the stochastic Allen--Cahn equation, 
become more involved recently (see, e.g., \cite{AC17,BGJK17,BJ16,BCH18,BG18, CHL16b,KLL18, LQ18, MP17,QW18, Wang18}) and are far from well-understood. 
 We refer to \cite{BGJK17,BJ16,BCH18,LQ18, MP17,QW18, Wang18} and references therein
for the strong convergence rate results of many different temporal and spatial approximations, and to \cite{BG18b,CH17} for the weak convergence rate results of temporal splitting type schemes. 
Up to now, there has been no essentially sharp weak convergence rate result of spatial approximation for parabolic SPDEs
with non-globally Lipschitz coefficients.
The present work makes  further contributions on the strong and weak convergence rates of spatial approximations for SPDEs with non-globally Lipschitz continuous nonlinearities but one-sided Lipschitz nonlinearities driven 
by additive noise.

Let $\OO=[0,L]$ and $\HH=L^2(\OO)$ be the real separable Hilbert space endowed with the usual inner product.  In this article, we mainly focus on the following semilinear parabolic SPDE,
\begin{align}\label{spde}
dX(t)+AX(t)dt&=F(X(t))dt+dW(t),\quad t\in [0,T],\\\nonumber
X(0)&=X_0,
\end{align}
where $0<T<\infty$, 
$-A$ is the Laplacian operator on  $\mathcal O$ under homogenous Dirichlet boundary condition, $F$ is the Nemytskii operator defined by $F(X)(\xi):=f(X(\xi)), \xi \in \OO$, where 
$f$ is a real-valued nonlinear function and satisfies Assumption \ref{as-dri}. In particular, Eq. \eqref{spde} is the stochastic Allen--Cahn equation if $F(X)=X-X^3$.
The stochastic process $\{W(t)\}_{t\ge 0}$ is a generalized $Q$-Wiener process on a filtered probability space $(\Omega,\mathcal F,\{\mathcal F_t\}_{t\ge 0},\PP) $. Under further assumptions on $X_0$, $Q$, $f$ and $\|A^{\frac {\beta -1}2}\|_{\LL_2^0}<\infty$,  $\beta\in (0,1]$,
similar arguments in \cite{BCH18,QW18} yield that
there is a 
unique mild solution $X$ of Eq. \eqref{spde}, which 
possesses the optimal spatial regularity $\E \Big[\|X(t)\|_{\HH^{\beta}}^p\Big]\le C(T,Q,X_0,p)$, $p\ge1$.  For the numerical study of SPDE with one-sided Lipschitz coefficient driven by the multiplicative noise under enough spatial regularity assumptions, we refer to \cite{FLZ17, MP17}. In this work, we do not consider the case of the multiplicative noise with low spatial regularity assumption, since it is more involved and beyond the scope of this article. 

One main contribution  of this article is applying the 
variational approach, combined with an appropriate error decomposition, to deduce 
the strong convergence rate of the spatial finite element method
for Eq. \eqref{spde} with one-sided Lipschitz coefficients under the mild assumption on $X_0$.
The corresponding finite element approximation $X^h$
satisfies  
\begin{align}\label{sge}
dX^h(t)+A_hX^h(t)dt&=P^hF(X^h(t))dt+P^hdW(t),\\\nonumber
X(0)&=X_0^h,
\end{align}
where $P^h$ is the  Galerkin finite element projection and $A_h$ is the discretization of $A$.
Recently, the authors in \cite{FLZ17}  prove strong convergence with sharp rates of the finite element method for stochastic Allen--Cahn equation with gradient-type multiplicative noise.
The authors in \cite{QW18} deduce the optimal strong convergence rate of the finite element method for stochastic Allen--Cahn equation driven by additive trace-class noise.
To the best of our knowledge, there exists no sharp strong convergence rate result of the finite element method approximating Eq. \eqref{spde} driven by general additive noise.
As the considered noise in Eq. \eqref{spde} could be very rough, a priori estimate of stochastic convolution is needed. We make use of the properties of 
$S^h$ and $P^h$ to get the non-uniform  estimate of the approximated stochastic convolution $Z^h$, and obtain the sharp strong convergence rate, for $X_0\in \mathcal C(\OO)$, $T>0$, $p\ge 1$,
\begin{align*}
\E\Big[\|X(T)-X^h(T)\|_{\HH}^p\Big]&\le C(X_0,T,p,\gamma)(1+T^{-\frac \gamma 2})^{p}h^{\gamma p},
\end{align*}
where $\gamma\le \beta$, if $\beta\in (\frac 12,1]$ and $\gamma <\beta $, if $\beta\in (0, \frac 12]$.
We remark that this approach to deduce the strong convergence rate of the numerical approximation is also available for more general case (see
Remark \ref{str-rk}).

Another main contribution is about the weak convergence rate of the finite element method 
for Eq. \eqref{spde} with one-sided Lipschitz coefficient.
In recent years, there already exist many different strategies on the weak error analysis for many different numerical schemes approximating  parabolic SPDEs with Lipschitz coefficients. We refer to e.g. \cite{AL16, bd18,Deb11,KLL12,KLL13} for the error
analysis  based on the associated Kolmogorov equation, to e.g. \cite{CJK14, HJK16} for applying the mild It\^o formula approach  and to e.g. \cite{AKL16, Wang16}  for other techniques. However,
no essentially sharp weak convergence rate of spatial approximation is established for parabolic SPDEs
with non-globally Lipschitz coefficients. 
There are three key points  to deduce the weak convergence rate of numerical approximations for Eq. \eqref{spde} with non-sided Lipschitz coefficients:  to give the regularity estimates of the corresponding Kolmogorov equation, to deduce the uniform estimate of the spatial approximation and  to get rid of the irregular terms in the weak error estimate.
Inspired by \cite{BG18b} where the authors shows the weak convergence order of the two temporal splitting type schemes approximating stochastic Allen--Cahn equation driven by space-time white noise,
we propose a regularizing procedure through a splitting strategy. 
Then we utilize the properties of $S^h$, $P^h$
and $A_h$ (see Section \ref{sec-2}), as well as the non-uniform estimate of the approximated stochastic convolution $Z^h$ to get a priori 
estimate of the finite element approximation.
At last, by using  the Malliavin integration by parts, together with 
the regularity estimates of the regularized Kolmogorov equation and a priori 
estimate of the finite element approximation, we derive 
the essentially sharp weak convergence rate result of $X^h$, for 
$\phi \in \mathcal C^2_b(\HH)$, $X_0\in \mathcal C(\OO)$, $T>0$, $\gamma<\beta$,
\begin{align*}
\Big|\E\Big[\phi(X(T))-\phi(X^h(T))\Big]\Big|
\le C(X_0,T,\gamma,\phi)(1+T^{-\gamma})h^{2\gamma}.
\end{align*}

The outline of this paper is as follows. In the next section,
some preliminaries are listed.
Section \ref{sec-str} is devoted to giving the  a priori estimates of Eq. \eqref{spde}, the strong convergence rate 
of the finite element method, as well as the a priori  estimates of the  finite element method and semi-discretized stochastic convolution.
In Section \ref{sec-wea}, we propose a new regularizing procedure
and give an approach to study the weak convergence rate of the finite element method by Malliavin calculus.

\section{Preliminaries}\label{sec-2}

In this section, we give  assumptions 
on $A$, $F$ and $W(t)$,  the abstract functional analytical framework of the considered equation and finite element method, and a brief introduction to Malliavin calculus.

Given two separable Hilbert spaces $(\mathcal H, \|\cdot \|_{\mathcal H})$ and $(\widetilde  H,\|\cdot \|_{\widetilde H})$, 
denote by  $\LL(\mathcal H, \widetilde H)$ and $\LL_1(\mathcal H, \widetilde H)$ the Banach spaces of all linear bounded operators 
and  the nuclear operators from $\mathcal H$ to $\widetilde H$, respectively. 
The trace of an operator $\mathcal T\in \LL_1(\mathcal H)$
is $\text{tr}[\mathcal T]=\sum_{k\in \N^+}\<\mathcal Tf_k,f_k\>_{\mathcal H}$, where $\{f_k\}_{k\in \N^+}$ ($\N^+=\{1,2,\cdots\}$) is any orthonormal basis of $\mathcal H$.
In particular, if $\mathcal T\ge 0$, $\text{tr}[\mathcal T]=\|\mathcal T\|_{\mathcal L_1}$.
Denote by $\LL_2(\mathcal H,\widetilde H)$ the space 
of Hilbert--Schmidt operators from $\mathcal H$ into $\widetilde H$, equipped with the usual norm given by  $\|\cdot\|_{\LL_2(\mathcal H,\widetilde H)}=(\sum_{k\in \N^+}\|\cdot f_k\|^2_{\widetilde H})^{\frac{1}{2}}$.
The following useful property and inequality (see e.g. \cite{AL16}) hold 
\begin{align}\label{tr}
\<\mathcal T,\mathcal S\>_{\mathcal L_2(\mathcal H,\widetilde H)}=\text{tr}[\mathcal T^*\mathcal S]=\text{tr}[\mathcal S\mathcal T^*],\quad \mathcal T, \;\;\mathcal S\in \mathcal L_2(\mathcal H,\widetilde H),\\\nonumber
|\text{tr}[\mathcal S\mathcal T^*]|\le \|\mathcal S\mathcal T^*\|_{\mathcal L_1}\le \|\mathcal S\|\|\mathcal T\|_{\mathcal L_1},\quad \mathcal S\in \LL(\mathcal H,\widetilde H), \;\; \mathcal T\in \LL_1(\mathcal H,\widetilde H),
\end{align}
where $\mathcal T^*$ is the adjoint operator of $\mathcal T$.
Let
$\mathcal C_b^k(\mathcal H,\widetilde H)$, $k\ \in\N^+$ be the space of $k$ times continuous  differentiable mappings from $\mathcal H$ to $\widetilde H$ with uniformly bounded Fr\'echet  derivatives up to order $k$. We endow $\mathcal C_b^k(\mathcal H,\widetilde H)$ with the seminorm $|\cdot|_{\mathcal C_b^k(\mathcal H,\widetilde H)}$, defined as for $g\in \mathcal C_b^k(\mathcal H,\widetilde H)$, $|g|_{\mathcal C_b^k(\mathcal H,\widetilde H)}$ is the smallest constant $C$ such that 
\begin{align*}
\sup_{x\in \mathcal H}\|D^m g(x)\cdot (\phi_1,\cdots, \phi_m)\|_{\widetilde H}
\le C\|\phi_1\|_{\mathcal H}\cdots\|\phi_m\|_{\mathcal H}, \; \forall \phi_1,\cdots\in \mathcal H, \phi_m \in \mathcal H, m\le k.
\end{align*}

Given a Banach space $(\mathcal E,\|\cdot\|_{\mathcal E})$, we denote by $\gamma( \mathcal H, \mathcal E)$ the space of $\gamma$-radonifying operators endowed with the norm defined by
$\|\mathcal T\|_{\gamma(\mathcal  H, \mathcal E)}=(\widetilde \E\|\sum_{k\in\N^+}\gamma_k \mathcal Tf_k \|^2_{\mathcal E})^{\frac 12}$,
where $\{\gamma_k\}_{k\in\N^+}$ is a Rademacher sequence on a
probability space $(\widetilde \Omega,\widetilde \FFF, \widetilde \PP)$.
For convenience, we denote  $\|\cdot\|=\|\cdot\|_{\HH}$ and $\<\cdot,\cdot\>=\<\cdot,\cdot\>_{\HH}$. Let  $L^q=L^q(\OO)$, $1\le q<\infty$ and 
$ E=\mathcal C(\OO)$ equipped with the usual  norms.
We also need the following Burkerholder inequality  in martingale-type 2 Banach spaces (see, e.g., \cite[Lemma 2.1]{VVW08}), for $L^q, q\in[2,\infty)$ and $p\ge 1$, there exists $C_{p,q}\in(0,\infty)>0$ such that
\begin{align}\label{Burk}
\Big\|\sup_{t\in [0,T]}\Big\|\int_0^t \phi(r)d\widetilde W(r)\Big\|_{L^q}\Big\|_{L^p(\Omega)}
&\le 
C_{p,q}\|\phi\|_{L^p(\Omega;L^2([0,T]; \gamma(\HH;L^q)))}\\\nonumber 
&\le  C_{p,q}\Big(\E\Big(\int_0^T\Big\|\sum_{k\in \N^+} (\phi(t) e_k)^2\Big\|_{L^{\frac q2}}dt\Big)^{\frac p2}\Big)^{\frac 1p}, 
 \end{align}
 where $\{\widetilde W(t)\}_{t\ge0}$ is the $\HH$-valued cylindrical 
 Wiener process and $\{e_k\}_{k\in \N^+}$ is an orthonormal basis of $\HH$.

\subsection{Main assumptions}
In this subsection, we introduce some useful notations
and  our main assumptions on $A$, $F$ and $W$.
Throughout this article, the initial datum $X_0$ is assumed to
 be a deterministic function and belongs to $E$ for convenience. We use $C$ to denote a generic constant, independent of  $h$, which differs from one place to another.

\begin{as}\label{as-lap}
Let $\OO=(0,L), L>0$ and $-A: D(A)\subset \HH \to \HH$ be the Laplacian operator under the homogeneous Dirichlet boundary condition, i.e., $Au=-\Delta u, u\in D(A)$.
\end{as}

Such assumption implies that
$-A$  generates 
an analytic and contraction $C_0$-semigroup $S(t), t\ge 0$ in $\HH$ and $L^q$.
It is also well known that the assumption on $\OO$ implies that the existence of the eigensystem $\{\lambda_k, e_k\}_{k\in\N^+}$ of 
$\HH$, such that $\lambda_k>0$, $Ae_k=\lambda_k e_k$ and $\lim\limits_{k\to \infty}\lambda_k=\infty$.
Let
 $\mathbb H^r$ be the Hilbert space equipped with the norm
 $\|\cdot \|_{\HH^r}:= \|A^{\frac {r}2}\cdot \|_{\HH}$ 
for  the fractional power $A^{\frac{r}{2}}, r\ge 0$.

\begin{as}\label{as-noi}
Let $W(t)$ be a Wiener process with covariance operator $Q$, where $Q$ is a bounded, linear, self-adjoint and positive definite operator on $\HH$. Assume that 
$\|A^{\frac {\beta-1}2 }\|_{\LL_2^0}<\infty$ with $0< \beta\le 1$,  where $\LL_2^0=\LL_2(U_0,\HH)$, $U_0=Q^{\frac 12}(\HH)$. In the case that $\beta\le \frac 12$, in addition assume that  $Q$ commutes with $A$.
\end{as}

\begin{as}\label{as-dri}
Let $K\in \N^+$ and $L_f>0$. Assume that $f:\R\to \R$ satisfies 
\begin{align*}
|f(\xi)|\le L_f(1+|\xi|^K), \quad f'(\xi)\le L_f, \quad |f'(\xi)|
\le L_f(1+|\xi|^{K-1}).
\end{align*} 
Let $F: L^{2K} \to \HH$ be the Nemytskii operator defined by
$F(X)(\xi)=f(X(\xi))$.

\end{as}

The above assumption ensures that $F: L^{2K}\to \HH$ satisfies for some constant $L=L(L_f,K)$, 
\begin{align*}
\<F(u)-F(v),u-v\>&\le L_f \|u-v\|^2,\; u,v \in L^{2K},\\
\|F(u)-F(v)\|&\le L(1+\|u\|_{E}^{K-1}+\|v\|_{E}^{K-1})\|u-v\|,\; u,v\in E, 
\end{align*}
where
$\|\cdot\|_{ E}$ is the supremum norm.
We remark that in the analysis of strong convergence rates, the assumption about the upper bound of the derivative of $f$ could be weakened to the monotone condition.
We also point out that  when studying the weak convergence rates,  more restricted condition on $F$ is needed.
The typical example of $f$ is a cubic polynomial
\begin{align*}
f(\xi)=a_3\xi^3+a_2\xi^2+a_1\xi+a_0,\;\; a_3<0, \;a_2,\;a_1,\;a_0\in \R.
\end{align*}
In this case, Eq. \eqref{spde} corresponds to  the  stochastic Allen--Cahn equation. 
We  remark that Assumptions \ref{as-lap}-\ref{as-dri} could be extend to $d\le 3$ and more general noise case (see Remark \ref{str-rk} in Section \ref{sec-str}). We also mention that 
the weak convergence rate of a full discretization of Eq. \eqref{spde} will be studied in \cite{CHS18}.

\subsection{Finite element method}
Let $(T_h)_{h\in (0,1)}$  be a quasi-uniform family
of  triangulations of $\OO$, i.e.,  $T_h$ is a partition of $\OO$, the parameter $h$ is the mesh size of $T_h$,  and the length of  each subinterval  
is bounded below by $ch$ for a constant $c>0$.
Let $(V_h)_{h\in(0,1)}$ be a family of spaces of 
continuous piecewise linear functions corresponding to $(T_h)_{h\in (0,1)}$, and $N_h$ be the dimension of $V_h$.
Denote $P^h:\HH\to V_h$ the orthogonal projection and $A_h: V_h\to V_h$ the discrete Laplacian satisfying $\<A_hu,v\>=\<\nabla u,\nabla v\>, u,v\in V_h$.
It is well known that the semi-discretization  $-A_hu^h=P^hf$ is finite element approximation of $-Au=f$
and that $\|u-u^h\|=\|A_h^{-1}P^hf-A^{-1}f\|\le Ch^2\|f\|$ (see e.g. \cite{Tho06}).
The operator $-A_h$ generates an analytic semigroup $(S^h(t))_{t\ge 0}$. In particular, there is an orthonormal  
eigenbasis  $\{e_i^h\}_{i=1}^{N_h}$ in $V^h$ equipped with the $\HH$ norm,  with eigenvalues $0<\lambda_1^h\le \lambda_2^h\le \cdots\le \lambda_{N_h}^h$ such that 
\begin{align*}
S^h(t)v^h=\sum_{i=1}^{N_h}e^{-\lambda_i^ht}\<v^h,e_i^h\>e_i^h, \quad v^h \in V_h, \;\; t\ge 0.
\end{align*}
We will often use the equivalence of the following two norms for $v^h\in V_h$, $\gamma \in [-\frac 12, \frac 12]$,
\begin{align}\label{eq-norm}
c\|A_h^{\gamma} v^h\|\le \|A^{\gamma} v^h\|\le C\|A_h^{\gamma} v^h\|,
\end{align}
the interpolation space $(\HH^{\beta}_h)_{\beta \in [-1,1]}$
and the properties of the Ritz projection $R^h: \HH^1 \to V_h$ and $P^h$ (see, e.g., \cite{AL16}, \cite[Chapter 3]{Tho06}),
\begin{align}\label{err-ord}
\|A^{\frac s2}(I-R^h)A^{-\frac r2}\|_{\LL(\HH)}&\le Ch^{r-s}, \quad 
0\le s\le 1\le r\le 2,\\\nonumber 
\|A^{\frac s2}(I-P^h)A^{-\frac r2}\|_{\LL(\HH)}&\le Ch^{r-s}, \quad 
0\le s\le 1, 0\le s\le r\le 2.
\end{align}

In the setting of strong convergence rate result,
we will need the error of the semigroups 
 $G^h(t):=S^h(t)P^h-S(t)$, $t\ge 0$, (see, e.g., \cite[Section 3]{QW18} or \cite[Chapter 3]{Tho06}) for $h\in (0,1]$,
\begin{align}\label{ord-str}
\|G^h(t)x\|&\le Ch^{u}t^{-\frac {u-v}2}\|x\|_{\HH^{v}},\; x\in\HH^v ,\; t>0, \; 0\le v\le u\le 2,\\\nonumber
\|G^h(t)x\|&\le Ct^{\frac \rho 2}\|x\|_{\HH^{-\rho}},\; x\in \HH^{-\rho}, t>0,\; 0\le \rho \le 1,\\\nonumber
\|G^h(t)x\|&\le Ct^{-1}h^{2-\rho}\|x\|_{\HH^{-\rho}},\; x\in \HH^{-\rho}, t>0,\; 0\le \rho\le 1,\\\nonumber
\|\int_0^tG^h(s)xds\|&\le Ch^{2-\rho}\|x\|_{\HH^{-\rho}},\; x\in \HH^{-\rho}
,\; t>0,\; 0\le \rho\le 1,
\\\nonumber
\Big(\int_0^t\|G^h(s)x\|^2ds\Big)^{\frac 12}
&\le Ch^{1+\rho}\|x\|_{\HH^\rho}, \; x\in \HH^{\rho},\;t>0,\; 0\le \rho\le 1.
\end{align}

Besides the above properties 
of finite element methods, the other important parts for our analysis are the smoothing effect of $S^h$ (see, e.g., \cite[Chapter 3]{Tho06})
\begin{align}\label{smo-ah}
\|A_h^{\gamma}S^h(t)P_h\|_{\LL(\HH)}&\le C_{\gamma}t^{-\gamma}, \;\;\gamma \ge 0,\; t >0,\\\nonumber
\int_0^t\|A_h^{\frac 12}S^hP_hx\|^2ds&\le C\|x\|^2,\;  x\in \HH,
\end{align}
and the boundedness of $P^h$ (see, e.g., \cite[Lemma 2.3]{TW83})
\begin{align*}
\|P^h\|_{\mathcal L(L^p)}\le C,\;\; 1\le p<\infty,\quad  \|P^h\|_{\mathcal L(E)}\le C.
\end{align*}
\subsection{Malliavin calculus}
In order to get the weak convergence rate, we  recall some preliminary about  Malliavin calculus in  Hilbert space (see e.g., \cite[Section 2]{AL16}), which will be used to deal with the singular term appeared in the weak error.
Since $Q$ is a bounded, linear, self-adjoint and positive definite operator on $\HH$, the corresponding Cameron-Martin space is $U_0=Q^{\frac 12}(\HH)$. Let $\mathcal  I: L^2([0,T]; U_0) \to L^2(\Omega)$ be an isonormal process, i.e., for any $\psi \in L^2([0,T]; U_0)$, $\mathcal I(\psi)$ is the centered Gaussian variable and $\E[\mathcal I(\psi_1)\mathcal I(\psi_2)]=\<\psi_1,\psi_2\>_{L^2([0,T]; U_0) }$, $\psi_1,\psi_2\in L^2([0,T]; U_0)$.
Let $\mathcal C_p^{\infty}(\R^N)$ be the space of all real-valued $\mathcal C^{\infty}$ functions on $\R^N$ with polynomial growth.
We denote the  family of smooth  real-valued cylindrical
random variables by
\begin{align*}
\mathcal S=\Big\{\mathcal X=g(\mathcal I(\psi_1),\cdots,\mathcal I(\psi_N)): g\in \mathcal C_p^{\infty}(\R^N), \psi_j \in L^2([0,T]; U_0), j=1,\cdots,N \Big\},
\end{align*}
and the  family of smooth  cylindrical $\HH$-valued 
random variables by
\begin{align*}
\mathcal S(\mathbb H)=\Big\{G=\sum_{i=1}^M\mathcal X_i\otimes h_i: \mathcal X_i \in \mathcal S, h_i\in \HH, M\ge 1\Big\}.
\end{align*} 
Then the Malliavin derivative of $G=\sum_{i= 1}^Mg_i(\mathcal I(\psi_1),\cdots,\mathcal I(\psi_N)\otimes h_i$ is defined by 
\begin{align*}
\mathcal D_sG=\sum_{i= 1}^M\sum_{j=1}^N\partial_j g_i(\mathcal I(\psi_1),\cdots,\mathcal I(\psi_N))\otimes (h_i\otimes \psi_j(s)).
\end{align*}
Since the derivative operator $\mathcal D$ is   closable (see, e.g., \cite[section 2]{AL16}), we denote $\mathbb D^{1,2}(\HH)$  the closure of $\mathcal S(\HH)$ with respect to Malliavin derivative equipped with the norm 
\begin{align*}
\|G\|_{\mathbb D^{1,2}(\HH)}=\Big(\E[\|G\|^2]+\E[\int_0^T\|\mathcal D_sG\|^2ds]\Big)^{\frac 12},
\end{align*}
where $\mathcal D_sG$ is the Malliavin derivative of $G$.
The key in the analysis of weak convergence rate is the following integration by parts formula(see, e.g., \cite[Section 2]{Deb11}).
For any random variable $G\in \mathbb D^{1,2}(\HH)$
and any predictable  process $\Theta \in L^2([0,T];\mathcal L_2^0)$, we have 
\begin{align}\label{int-by}
\E\Big[\Big\<\int_0^T\Theta (t)dW(t),G\Big\>\Big]
=\E\Big[ \int_0^T \Big\<\Theta (t),\mathcal D_tG\Big\>_{\mathcal L_2^0}dt\Big].
\end{align} 
Moreover, we also need the chain rule of the Malliavin derivative.
Let $\mathcal V$ be another separable Hilbert space and $\sigma \in \mathcal C_b^1(\HH,\mathcal V)$.
Then we have $\sigma(G)\in \mathbb D^{1,2}(\mathcal V)$,
\begin{align*}
\mathcal D_t^y(\sigma(G))&=\mathcal D\sigma(G)\cdot\mathcal D_t^yG, \quad y\in U_0, \quad G\in \mathbb D^{1,2}(\HH),\\
\mathcal D_t(\sigma(G))&=\mathcal D\sigma(G)\mathcal D_tG, \quad G\in \mathbb D^{1,2}(\HH),
\end{align*} 
where $\mathcal D_t^yG:=\mathcal D_tG y$ is the derivative of  $G$ in the direction of $y\in U_0$.

\section{A priori estimate and strong convergence rate}
\label{sec-str}

In this section, we present the strong convergence rate of  the finite element method, as well as the a priori estimate of the discrete stochastic convolution and  the finite element method.

\subsection{A priori estimate}
Combining the equivalence of Eq. \eqref{spde} and the 
following random PDE 
\begin{align*}
dY+AYdt&=F(Y+Z)dt,\;\; Y(0)=X_0,\\
dZ+AZdt&=dW(t),\;\; Z(0)=0,
\end{align*}
with the similar arguments in the proofs of 
\cite[Theorem 7.7]{DZ14} and \cite[Lemma 3.3]{BCH18}, we  
get the following a priori estimate on the exact solution 
of  Eq. \eqref{spde}.

\begin{lm}\label{exa}
Under Assumptions \ref{as-lap}-\ref{as-dri}, there exists a unique mild solution $X$ of Eq. \eqref{spde}.
Moreover,  for $t\in (0,T]$, $p\ge 1$, there exists a
constant $C(T,p)>0$ such that 
\begin{align*}
\sup_{s\in[0,t]}\E\Big[\|X(s)\|_{E}^p\Big]&\le C(T,p)(1+\|X_0\|_E^p),
\\
\E\Big[\|X(t)\|_{\HH^{\beta}}^p\Big]&\le (1+t^{-\frac {\beta p} 2})C(T,p)(1+\|X_0\|^p).
\end{align*}
\end{lm}

Now, we are in a position to derive a priori estimate for the semi-discretization Eq. \eqref{sge}. At first, we prove  the smoothing  property of $S^h(t),\; t\ge 0$.

\begin{lm}\label{smo0}
For $t>0$ and $2\le p\le \infty$, there exists 
 a positive constant $C$  such that for $f\in \HH$, 
\begin{align*}
\|S^h(t)P^hf\|_{L^p}\le Ct^{-\frac 12(\frac 12 -\frac 1p)}\|f\|.
\end{align*}
\end{lm}

\begin{proof}
Since $S^h(t)P^hf\in V^h$, we have 
\begin{align*}
S^h(t)P^hf= \sum_{i=1}^{N^h}e^{-\lambda_i^h t}\<f,e_i^h\>e_i^h.
\end{align*}
Then the uniform boundness  of $e_i^h$ and $ci^2 \le \lambda _i^h\le Ci^{2}$, $1\le i\le N^h$ in \cite[Section]{AL16} yield that 
\begin{align*}
\|S^h(t)P^hf\|_{E}&=\|\sum_{i=1}^{N^h}e^{-\lambda_i^h t}\<f,e_i^h\>e_i^h\|_{E} \le (\sum_{i=1}^{N^h}e^{-2\lambda_i^h t})^{\frac 12}\|f\|\le Ct^{-\frac 14}\|f\|,
\end{align*}
and 
\begin{align*}
\|S^h(t)P^hf\|&\le \big\|\sum_{i=1}^{N^h}e^{-\lambda_i^h t}\<f,e_i^h\>e_i^h\big\|\le C\|f\|.
\end{align*}
The Riesz--Thorin interpolation theorem (see e.g. \cite{Ste56}) leads to the desired result. 
\end{proof}

The other tool to get the a priori estimate is the 
weak discrete maximum principle in  \cite[Lemma 3.4]{CLT94}.

\begin{lm}\label{smo}
Under the assumptions on $T^h$ and $V_h$,
there exists a positive constant $C$  such that, for any $v^h\in V_h$,
\begin{align}\label{wdmp}
\|S^h(t)v^h\|_{L^{\infty}}\le C\|v^h\|_{L^{\infty}}, \;\; t>0.
\end{align}
\end{lm}

We remark that in the case of  higher dimension, the similar boundedness  results of  finite element methods  still hold (see e.g. \cite[Chapter 6]{Tho06}).
Next, we give the a priori estimate of the semi-discretized  
stochastic convolution $Z^h$, which satisfies 
\begin{align*}
dZ^h(t)+A_hZ^h(t)=P^hdW(t), \quad Z^h(0)=0.
\end{align*} 

\begin{lm}\label{pri-noi}
Let $\mathcal V=E $ or $L^{2q}$ $(q\ge 1)$.
Under Assumptions \ref{as-lap}-\ref{as-noi}, 
there exists a constant $C(T,p)>0$ such that  the discretized  stochastic convolution $Z^h$ satisfies that for $t\in (0,T]$ and $p\ge 1$, 
\begin{align*}
\E\Big[\|Z^h(t)\|_{\mathcal V}^p\Big]\le C(T,p), \quad \text{if} \quad \beta>\frac 12,\\
 \E\Big[\|Z^h(t)\|_{\mathcal V}^p\Big]\le C(T,p)(1+\log(\frac 1h))^{\frac p2}, \quad \text{if} \quad 0<\beta\le \frac 12.
\end{align*}
Moreover, if $Q=I$, $\beta \in [0,\frac 12)$, then
for $t\in (0,T]$ and $p\ge 1$, there exists a constant $C(T,p)$ such that 
\begin{align*}
\E\Big[\|Z^h(t)\|_{\mathcal V}^p\Big]\le C(T,p). 
\end{align*}
\end{lm}

\begin{proof}
The a priori estimate in  
the case that $\beta>\frac 12$ is directly proven by 
using the Sobolev embedding theorem, the Burkholder inequality, and the smoothing property of $A_h$ \eqref{smo-ah}.
Now we focus on the case $\beta\le \frac 12$ and take $\mathcal V=L^{2q}$, $q\ge 1$ as example. Similar arguments yield  the case $\mathcal V= E$.
Notice that  $Z^h(t,\xi)=\sum_{k\in \N^+}\int_0^t 
\sum_{i=1}^{N^h} e^{-\lambda_i^h s} \<\sqrt{q_k}e_k, e_i^h\>e_i^h(\xi)d\beta_k(s)$, where $\{e_{k},q_k\}_{k\in \N^+}$ is the eigensystem of $Q$. The Fubini theorem, Fourier transform, Burkholder inequality \eqref{Burk} and uniform boundedness  of $e_i^h, 1\le i\le N^h$ in $E$ (see e.g. \cite[Appendix]{Wal05})  yield that 
\begin{align*}
\E\Big[\|Z^h(t)\|_{L^{2q}}^p\Big]
&\le C\E \Big[\Big(\int_0^t \Big\|\sum_{k\in\N^+}  \big(\sum_{i=1}^{N^h} e^{-\lambda_i^h s} \<\sqrt{q_k}e_k, e_i^h\>e_i^h\big)^2 \Big\|_{L^{q}}ds\Big)^{\frac p2}\Big]\\
&\le 
C\E \Big[\Big(\int_0^t \Big\|  \sum_{i,j=1}^{N^h} e^{-(\lambda_i^h+\lambda_j^h)s} \<Q^{\frac 12}e_i^h, Q^{\frac 12}e_j^h\>e_i^he_j^h \Big\|_{L^{q}}ds\Big)^{\frac p2}\Big]\\
&\le 
C\E \Big[\Big(\int_0^t \sum_{i,j=1}^{N^h}e^{-(\lambda_i^h+\lambda_j^h)s} \|e_i^h\|_E\|e_i^h\| \|e_j^h\|_E\|e_j^h\|ds\Big)^{\frac p2}\Big]\\
&\le C\E \Big[\Big(\int_0^t \sum_{i,j=1}^{N^h}e^{-(\lambda_i^h+\lambda_j^h)s}ds\Big)^{\frac p2}\Big].
\end{align*}
By $ci^2 \le \lambda _i^h\le Ci^{2}, 1\le i\le N^h$ and $N^h\le \mathcal O(\frac 1h)$, we have 
\begin{align*}
\E\Big[\|Z^h(t)\|_{L^{2q}}^p\Big]
&\le C\E \Big[\Big(\int_0^t \sum_{i,j=1}^{N^h}e^{-(\lambda_i^h+\lambda_j^h)s}ds\Big)^{\frac p2}\Big]\le C \E \Big[\Big(\int_0^t \sum_{i=1}^{(N^h)^2}e^{-ics}ds\Big)^{\frac p2}\Big]\\
&\le C \E \Big[\Big(\int_0^{h^{l}}\frac 1{h^2}ds\Big)^{\frac p2}\Big]+C\E \Big[\Big(\int_{h^l}^t \int_1^{\infty}e^{-c\xi s}d\xi ds\Big)^{\frac p2}\Big]\\
&\le C(h^{l-2}+(\log(1+t)+\log(\frac 1 h))^{\frac p2})\\
&\le C(1+(\log(1+t))^{\frac p2}+(\log(\frac 1h))^{\frac p2}),
\end{align*}
for a large $l\in \N^+$.
In particular, if $Q=I$, then the logarithmic factor can be eliminated  as
\begin{align*}
\E\Big[\|Z^h(t)\|_{L^{2q}}^p\Big]
&\le C\E \Big[\Big(\int_0^t \Big\|\sum_{k\in\N^+}  \big(\sum_{i=1}^{N^h} e^{-\lambda_i^h s} \<e_k, e_i^h\>e_i^h\big)^2 \Big\|_{L^{q}}ds\Big)^{\frac p2}\Big]\\
&\le 
C\E \Big[\Big(\int_0^t \Big\|\sum_{k\in\N^+}  \sum_{i,j=1}^{N^h} e^{-(\lambda_i^h+\lambda_j^h)s} \<e_k, e_i^h\>\<e_k, e_j^h\>e_i^he_j^h \Big\|_{L^{q}}ds\Big)^{\frac p2}\Big]\\
&\le 
C\E \Big[\Big(\int_0^t \Big\|  \sum_{i,j=1}^{N^h} e^{-(\lambda_i^h+\lambda_j^h)s} \<e_i^h, e_j^h\>e_i^he_j^h \Big\|_{L^{q}}ds\Big)^{\frac p2}\Big]\\
&\le 
C\E \Big[\Big(\int_0^t \Big\|  \sum_{i=1}^{N^h} e^{-2\lambda_i^hs} (e_i^h)^2\Big\|_{L^{q}}ds\Big)^{\frac p2}\Big]
\le Ct^{\frac p4}.
\end{align*}
Summing up all the estimates, we finish the proof.
\end{proof}

The following a priori estimate is very useful for deducing 
the weak convergence rate in Section \ref{sec-wea} and has its own interest.

\begin{prop}\label{pri-fem}
Let $\mathcal V=E$ or $ L^{2q}$ $(q\ge 1)$.
Under Assumptions \ref{as-lap}-\ref{as-dri} with $K<5$, there exists $C(p,T,X_0)>0$ such that the  unique mild solution $X^h$ of Eq. \eqref{sge}  satisfies 
\begin{align*}
\sup_{t\in[0,T]}\E\Big[\|X^h(t)\|_{\mathcal V}^p\Big]&\le C(p,T,X_0)(1+(\log(\frac 1h)))^{\frac {K^2p} 2},\;\; \text{for} \;\; p\ge 1.
\end{align*}
\end{prop}
\begin{proof}
For convenience, we only prove the case $\mathcal V=E$.
By the equivalence of the stochastic PDE 
\begin{align*}
dX^h+A_hX^hdt=P^hF(X^h)dt+P^hdW(t), X^h(0)=P^hX_0,
\end{align*} 
and the random PDE
\begin{align*}
dY^h+A_hY^hdt&=P^hF(Y^h+Z^h)dt, Y^h(0)=P^hX_0,\\
dZ^h+A_hZ^hdt&=P^hdW(t),Z^h(0)=0
\end{align*}
and Lemma \ref{pri-noi}, it suffices to bound $\E\Big[\|Y^h(t)\|_{E}^p\Big]$. 
The higher regularity of $Y^h$ and  the dissipativity  of  $F$ imply  that
\begin{align*}
&\|Y^h(t)\|^2+2\int_0^t\<\nabla Y^h,\nabla Y^h\>ds
\\
&=
\|X^h(0)\|^2+2\int_0^t\<F(Y^h+Z^h),Y^h\>ds\\
&= \|X^h(0)\|^2+
2\int_0^t\<F(Y^h+Z^h)-F(Z^h),Y^h\>ds
+2\int_0^t\<F(Z^h),Y^h\>ds\\
&\le \|X^h(0)\|^2+
C\int_0^t\|Y^h\|^2ds
+C\int_0^t(1+\|Z^h\|_{L^{2K}}^{2K})ds.
\end{align*}
The $p$-moment boundedness of $\|Z^h\|_{E}$ and the Gronwall's inequality yield that for $p\ge 1$,
\begin{align*}
&\sup_{t\in [0,T]}\|Y^h(t)\|^{2p}
+\big(\int_0^T\|\nabla Y^h\|^2ds\big)^{p}\\
&\le 
C(p,T)(1+\|X^h(0)\|^{2p}+\int_0^T\|Z^h\|_{L^{2K}}^{2Kp}ds).
\end{align*}
Next, based on the above estimates and Lemma \ref{pri-noi}, we are in position to
prove the desired result.
The mild form of $Y^h$, Lemmas \ref{smo0} and \ref{smo}, together with the Gagliardo--Nirenberg--Sobolev inequality $\|f\|_{L^{2K}}\le C\|\nabla f\|^{\frac {K-1}{2K}}\|f\|^{\frac {K+1}{2K}}$, lead to
\begin{align*}
\|Y^h(t)\|_{E}&\le \|S^h(t)P^hX_0\|_E+\int_0^t\|S^h(t-s)P^hF(Y^h+Z^h)\|_Eds\\
&\le C\|X_0\|_E
+C\int_0^t(t-s)^{-\frac 14}\|F(Y^h+Z^h)\|ds\\
&\le C\|X_0\|_E
+C\int_0^t(t-s)^{-\frac 14}(1+\|Y^h\|_{L^{2K}}^K+\|Z^h\|_{L^{2K}}^{K})ds\\
&\le  C\|X_0\|_E
+C\int_0^t(t-s)^{-\frac 14}(1+\|Z^h\|_{L^{2K}}^K)ds\\
&\quad+C\int_0^t(t-s)^{-\frac 14}\|\nabla Y^h\|^{\frac {K-1}2}\|Y^h\|^{\frac {K+1}2}ds.
\end{align*}
Taking the $p$th moment on both sides, together with the H\"older and Young inequalities and the boundedness of $\int_0^T\|\nabla Y^h\|^2ds$ and $Z^h$, yields that for $p\ge 1$ and  $K<5$,
\begin{align*}
&\E\Big[\sup_{t\in[0,T]}\|Y^h(t)\|_{E}^p\Big]\\
&\le C\|X_0\|_E^p+C(T,p)(1+\log(\frac 1h))^{\frac {Kp}2}\\
&\quad+C\E\Big[(\int_0^t(t-s)^{-\frac 14}\|\nabla Y^h\|^{\frac {K-1}2}ds)^p\sup_{s\in [0,T]}\|Y^h(s)\|^{\frac {(K+1)p}2}\Big]\\
&\le C\sqrt{\E\Big[\big(\sup_{t\in[0,T]}\int_0^t(t-s)^{-\frac 14}\|\nabla Y^h\|^{\frac {K-1}2}ds\big)^{2p}\Big]}
\sqrt{\E\Big[\sup_{s\in [0,T]}\|Y^h(s)\|^{(K+1)p}\Big]}\\
&\quad+C(X_0,T,p)(1+\log(\frac 1h))^{\frac {Kp}2}\\
&\le C(X_0,T,p)(1+\log(\frac 1h))^{\frac {(K+1)Kp}4}\sqrt{\E\Big[\big(\sup_{t\in[0,T]}\int_0^t\|\nabla Y^h\|^2ds\big)^{\frac {(K-1)p}2}\Big]}\\
&\quad+C(X_0,T,p)(1+\log(\frac 1h))^{\frac {Kp}2}\\
&\le C(X_0,T,p)(1+\log(\frac 1h))^{\frac {K^2p}2}.
\end{align*}
Using the fact that
$$ 
\E\Big[\sup\limits_{t\in[0,T]}\|X^h(t)\|_E^p\Big]\le C_p
\E\Big[\sup\limits_{t\in[0,T]}\|Y^h\|_E^p\Big]+C_p\E\Big[\sup\limits_{t\in[0,T]}\|Z^h(t)\|_E^p\Big],$$ together with the a priori estimate on $Y^h$ above and $Z^h$ in Lemma \ref{pri-noi}, we complete the proof.
\end{proof}

\begin{rk}
Under the assumptions of Proposition \ref{pri-fem},
if  in addition  assume that $\|A^{\frac {\beta -1}2}\|_{\LL_2^0}<\infty,\beta>\frac 12$ or that $Q=I$,
we have the following  optimal estimate 
\begin{align*}
\sup_{t\in[0,T]}\E\Big[\|X^h(t)\|_{\mathcal V}^p\Big]&\le C(p,T,X_0).
\end{align*}
The above a priori estimate of $X^h$ in $\mathcal V$ is
a crucial part to derive the weak convergence rate in Section \ref{sec-wea}. 
This is the main reason why we require that Assumption \ref{as-dri} holds for $K<5$ in Proposition \ref{pri-fem}.
\end{rk}

\subsection{Strong convergence rate}

In this subsection, we aim to give the strong convergence result of the finite element method.
 We also remark that this approach to get strong convergence rates does not require 
the additional a priori estimate of the spatial approximation $X^h$.

\begin{tm}\label{strong}
Under Assumptions \ref{as-lap}-\ref{as-dri}, the finite element approximation $X^h(t)$ is strongly convergent to $X(t)$, $t\in (0,T]$ and satisfies, for $p\ge 1$,  
\begin{align*}
\E\Big[\|X(t)-X^h(t)\|^p\Big]&\le C(X_0,T,p)(1+t^{-\frac \beta 2})^{p}h^{\beta p},\;\; \text{for} \;
\beta>\frac 12,\\
\E\Big[\|X(t)-X^h(t)\|^p\Big]&\le C(X_0,T,p)(1+t^{-\frac {\beta}2}+
(\log(\frac 1h))^{\frac {(K-1)}2})^ph^{\beta p},\;\;\text{for}\; \beta \le \frac 12.
\end{align*}
\end{tm}
\begin{proof}
Since $A$  does not commute with $P^h$, we could not  use the usual strategy which divides  the strong error 
$X(t)-X^h(t)$ into $(I-P^h)X(t)$ and $P^hX(t)-X^h(t)$.
Thus, we introduce a new auxiliary process $\widetilde Y^h$ which satisfies 
\begin{align*}
d\widetilde Y^h+A_h\widetilde Y^hdt=P^hF(Y+Z)dt, \quad \widetilde Y^h(0)=X^h(0).
\end{align*}
Now, we split strong error as 
\begin{align*}
X(t)-X^h(t)&=Y(t)-Y^h(t)+Z(t)-Z^h(t)\\
&=(Y(t)-\widetilde Y^h(t))+(\widetilde Y^h(t)-Y^h(t))+(Z(t)-Z^h(t)),
\end{align*}
and estimate the three terms, respectively.
Using 
the estimates \eqref{err-ord} of $G^h(t):=S^h(t)P^h-S(t)$, $t\ge 0$ and Burkholder inequality, we get 
\begin{align*}
&\E\Big[\|Z(t)-Z^h(t)\|^p\Big]\\
&\le C_p\E\Big[\Big(\int_0^t\|G^h(t-s)\|^2_{\LL_2^0}ds\Big)^{\frac p2}\Big]\\
&\le C_p\E\Big[\Big(\int_0^t\|(R^h-I)S(t-s)\|^2_{\LL_2^0}ds\Big)^{\frac p2}\Big]\\
&\le C_p\E\Big[\Big(\int_0^t\|(R^h-I)A^{-\frac \beta 2}\|_{\LL}^2\|A^{\frac 12}S(t-s)A^{\frac {\beta -1}2}\|^2_{\LL_2^0}ds\Big)^{\frac p2}\Big]\\
&\le C(T,p)\|A^{\frac {\beta -1}2}\|_{\LL_2^0}^ph^{\beta p}.
\end{align*} 
The mild forms of $Y$ and $\widetilde Y^h$, together with the a priori estimates of $Y$ and $Z$ and the properties  \eqref{ord-str} of finite element method, yield that for $0\le u <2$,
\begin{align*}
&\E\Big[\|Y(t)-\widetilde Y^h(t)\|^p\Big]\\
&\le 
C_p\E\Big[\|G^h(t)X_0\|^p\Big]
+C_p\E\Big[\Big\|\int_0^tG^h(t-s)F(Y+Z)ds\Big\|^p\Big]\\
&\le 
C_p h^{u p}t^{-\frac{ u p} 2}\|X_0\|^p
+
C_p
\E\Big[\Big\|\int_0^tG^h(t-s)F(Y+Z)ds\Big\|^p\Big]\\
&\le 
C_p h^{u p}t^{-\frac {u p} 2}\|X_0\|^p
+
C(p,T)
\E\Big[\Big(\int_0^t h^{u}(t-s)^{-\frac u 2}\|F(Y(s)+Z(s))\|ds\Big)^p\Big]\\
&\le 
C_p h^{u p}t^{-\frac {u p} 2}\|X_0\|^p
+
C(p,T)h^{u p}\E\Big[\sup_{s\in [0,T]}\Big(1+\|Y(s)\|^{K}_{L^{2K}}+\|Z(s)\|^K_{L^{2K}}\Big)^p\Big]\\
&\le C(X_0,T,p) h^{u p} (1+t^{-\frac {u p} 2}).
\end{align*}
Notice that the similar arguments as in the proof of \cite[ Lemma 3.1]{BCH18} yield that  $\E\Big[\|\widetilde Y^h(t)\|_E^p\Big]\le C(p,T,X_0)$.
Next we deal with the term $\widetilde Y^h(t)-Y^h(t)$.
The random PDE forms of $\widetilde Y^h(t)$ and $Y^h(t)$
lead to
\begin{align*}
&\|\widetilde Y^h(t)- Y^h(t)\|^2\\
&\le 
-2\int_0^t \|\nabla(\widetilde Y^h(s)- Y^h(s)) \|^2ds\\
&\quad+\int_0^t2\<F(Y(s)+Z(s))-F(Y^h(s)+Z^h(s)), \widetilde Y^h(s)- Y^h(s)\>ds\\
&\le \int_0^t2\<F(Y(s)+Z(s))-F(\widetilde Y^h(s)+Z^h(s)), \widetilde Y^h(s)- Y^h(s)\>ds\\
&\quad+\int_0^t2\<F(\widetilde Y^h(s)+Z^h(s))- F(Y^h(s)+Z^h(s)), \widetilde Y^h(s)- Y^h(s)\>ds.
\end{align*}
For $\beta\le \frac 12$, by the monotonicity of $F$ and Lemma \ref{pri-noi}, we have 
\begin{align*}
&\sup_{s\in [0,t]}\|\widetilde Y^h(s)- Y^h(s)\|^{2}\\
&\le 
C\int_0^t \sup_{s\in [0,r]}\|\widetilde Y^h(s)- Y^h(s)\|^{2}dr
+C\Big(\int_0^t\big(\|Y(r)-\widetilde Y^h(r)\|
+\|Z(r)-Z^h(r)\|\big)\big(1+\\
&\quad \quad\|Y(r)\|_E^{K-1}+\|\widetilde Y^h(r)\|_E^{K-1}+\|Z(r)\|_E^{K-1}+\|Z^h(r)\|^{K-1}_E\big)\|\widetilde Y^h(r)- Y^h(r)\|dr\Big)\\
&\le
C\int_0^t \sup_{s\in [0,r]}\|\widetilde Y^h(s)- Y^h(s)\|^{2}dr
+\epsilon \sup_{r\in [0,t]}\|\widetilde Y^h(r)- Y^h(r)\|^2
\\
&\quad+C(\epsilon)\Big(\int_0^t\big(\|Y(r)-\widetilde Y^h(r)\|
+\|Z(r)-Z^h(r)\|\big)\\
&\quad \quad\big(1+\|Y(r)\|_E^{K-1}+\|\widetilde Y^h(r)\|_E^{K-1}+\|Z(r)\|_E^{K-1}+\|Z^h(r)\|^{K-1}_E\big)dr\Big)^2.
\end{align*}
Taking the $p$th moment yields that  for $0 \le \mu <2$,
\begin{align*} 
&\E\Big[\sup_{s\in [0,t]}\|\widetilde Y^h(s)- Y^h(s)\|^{2p}\Big]\\
&\le C_p\int_0^t \E \Big[\sup_{s\in [0,r]}\|\widetilde Y^h(s)- Y^h(s)\|^{2p}\Big]dr
+C(X_0,T,p)\Big(1+\log(\frac 1h)\Big)^{(K-1)p}
\\
&\qquad\times\Big(\Big(\int_0^t h^{\mu}(1+s^{-\frac \mu2})ds\Big)^{2p}
+h^{2\beta p}\Big)\\
&\le C\int_0^t\E\Big[\sup_{s\in [0,r]}\|\widetilde Y^h(s)- Y^h(s)\|^{2p}\Big]dr+C(X_0,T,p)\Big(1+\log(\frac 1h)\Big)^{(K-1)p}
h^{2\beta p}.
\end{align*}
Then the Gronwall inequality leads to
\begin{align*}
\E\Big[\|\widetilde Y^h(t)- Y^h(t)\|^{2p}\Big]
&\le C(X_0,T,p)(1+\log(\frac 1h))^{(K-1)p}h^{2\beta p}.
\end{align*}
For $\beta>\frac 12$, similar arguments, together with  
 Lemma \ref{pri-noi} and the boundedness of $\|\widetilde Y^h(t)\|_E$ imply that 
\begin{align*}
\E\Big[\|\widetilde Y^h(t)- Y^h(t)\|^{2p}\Big]
&\le C(X_0,T,p)h^{2\beta p}.
\end{align*}
Combining the strong error  estimates of $Y(t)-\widetilde Y^h(t)$, $\widetilde Y^h(t)-Y^h(t)$ and $Z(t)-Z^h(t)$ together, we finish the proof.
\end{proof}

\begin{rk}
Under the assumptions of Theorem \ref{strong}, if in addition  $X_0\in \HH^{\beta }$, then the term $t^{-\frac \beta 2}$ in the strong convergence rate result can be eliminated.
When $Q=I$,  the logarithmic factor in the strong error estimate can also be eliminated. We also remark that the approach to deduce strong convergence rates of numerical schemes is also available for SPDEs with  non-monotone coefficients (see e.g. \cite{CHS18a}).
\end{rk}

\begin{rk}\label{str-rk}
Assume that 
$\OO\subset \mathcal R^d, d\le 3$ is a bounded open domain with smooth boundary, 
$\{W(t)\}_{t\ge0}$ satisfies 
$\|A^{\frac {\beta-1}2}\|_{\LL_2^0}<\infty$ for some $\beta \in [\frac {(K-1)d}{2K},2]$, and $\sup_{t\in[0,T]}\|\int_0^tS(t-s)dW(s)\|_{L^{p_0}(\Omega,L^{2K^2})}<\infty$ for a sufficient large number $p_0\in \N^+$. Then for $X_0\in \HH^{\beta}$, $p\ge1$, it holds that 
\begin{align*}
\E\Big[\|X(t)-X^h(t)\|^p\Big]\le C(X_0,T,p)h^{\beta p}.
\end{align*}
In this case, 
the proof of the strong convergence rate result
does not rely on the a priori estimates of the finite element method. The key ingredients lie on using the Sobolev embedding $\HH^{\gamma} \hookrightarrow L^{2K}, \gamma \in [\frac {(K-1)d}{2K}, 1]$ and the dissipativity of $-A$ and $-A_h$.
\end{rk}

\section{Regularity of Kolmogorov equation and weak convergence rate}
\label{sec-wea}

\subsection{Regularity of Kolmogorov equation}
Denote $\mathcal C^2_b(\HH):=\mathcal C^2_b(\HH,\R)$.  Set $U(t,x)=\E[\phi(X(t,x))]$, then formally, $U$
is the solution of the Kolmogorov equation associated with Eq. \eqref{spde}:
\begin{align*}
\frac {\partial U(t,x)}{\partial t}
=\<-Ax+F(x),D U(t,x)\>+\text{tr}[Q^{\frac 12}D^2U(t,x)Q^{\frac 12}].
\end{align*}
To give rigorous meaning of the Kolmogorov equation, we follow the approach in \cite{BG18b}. We first apply the splitting strategy inspired by  \cite{BCH18,BG18} to  regularize the original equation  and get a regularized 
Kolmogorov equation.
Then making use of the regularity of the regularized 
Kolmogorov equation and  integration by parts formula in Malliavin sense, we obtain the weak convergence rate of the finite element method.

Now, we are in a position to give the rigorous meaning of the regularized Kolmogorov equation and
regularity estimates of $D U$ and $D^2U$.
The following lemma   is useful in constructing the 
regularized PDE and its corresponding Kolmogorov equation. For a function $f$ on $\R$, we denote  the first derivative and second derivative by $f'$ 
and $f''$.

\begin{lm}\label{ode-lm}
Let $L_f>0$, $K\in \N^+$ and $f:\R\to \R$ satisfy 
\begin{align*}
|f(\xi)|&\le L_f(1+|\xi|^K), \quad |f'(\xi)|
\le L_f(1+|\xi|^{K-1}),\\
f'(\xi)& \le L_f,\quad |f''(\xi)|
\le L_f(1+|\xi|^{(K-2)\lor0}).
\end{align*} 
Then the phase flow  $\Phi_{t}$
of the  differential equation 
\begin{align}\label{ode}
dx(t)=f(x(t))dt,\quad x(0)=\xi \in \R,
\end{align} 
satisfies for all $\xi \in \R,$
\begin{align*}
|\Phi_t(\xi)|\le C(f,t)(1+|\xi|),\quad \Phi_t'(\xi)\le C(f,t), \quad |\Phi_t''(\xi)|\le C(f,t)(1+|\xi|^{(K-2)\lor 0}). 
\end{align*}
\end{lm}
\begin{proof}
From the properties of $f$ and the Young inequality, 
it follows that
\begin{align*}
|x(t)|^2
&= |\xi|^2+\int_0^t 2(f(x(s))-f(0))x(s)ds+\int_0^t 2f(0)x(s)ds\\
&\le |\xi|^2+ \int_0^t(L_f^2+(1+2L_f) |x(s)|^2)ds.
\end{align*}
Then  Gronwall's inequality implies that
$|\Phi_t(\xi)|=|x(t)|\le C(f,t)(1+|\xi|)$.
Similarly, using the differentiable dependence on initial data, we obtain 
\begin{align*}
\Phi_t'(\xi)&= 1+\int_0^t f'(\Phi_s(\xi))\Phi_s'(\xi)ds,
\end{align*}
which, together with  Gronwall's inequality, yields that 
$0\le \Phi_t'(\xi)\le e^{L_ft}.$
Similar arguments lead to
\begin{align*}
|\Phi_t''(\xi)|^2&\le 2\int_0^t f'(\Phi_s(\xi))|\Phi_s''(\xi)|^2ds
+2\int_0^t f''(\Phi_s(\xi))(\Phi_s'(\xi))^2\Phi_s''(\xi)ds\\
&\le \int_0^t(2L_f+1)|\Phi_s''(\xi)|^2ds
+\int_0^t |f''(\Phi_s(\xi))|^2|\Phi_s'(\xi)|^4ds,
\end{align*}
which indicates that $|\Phi_t''(\xi)|\le C(f,t)(1+|\xi|^{(K-2)\lor0}).$
\end{proof}

With the help of  Lemma \ref{ode-lm}, we introduce our regularizing procedures. Based on the strategy of splitting method in \cite{BG18}, we split the Eq. \eqref{spde} into
two sub-systems
\begin{align*}
dX_1=F(X_1)dt,\quad 
dX_2=-AX_2dt+dW(t).
\end{align*}
Then given a fixed time step size $\delta t>0$, the splitting method is defined as
\begin{align*}
\widetilde X_{n+1}&:=S(\delta t)\Phi_{\delta t}(\widetilde X_n)
+\int_{t_n}^{t_{n+1}}S(t_{n+1}-s)dW(s)\\
 &=S_{\delta t}\widetilde X_n+\delta t S_{\delta t} \Psi_{\delta t}(\widetilde X_n)
+\int_{t_n}^{t_{n+1}}S(t_{n+1}-s)dW(s),
\end{align*}
 for $ 0\le n \le N-1, N\delta t=T, t_n=n \delta t $
where $\Psi_{t}(x):=\frac {\Phi_t(x)-x}t$, $t>0$
and $\Psi_{0}(x)=F(x)$. 
Notice that the splitting method can be used to approximate SPDE with non-monotone coefficients in strong and weak convergence senses (see, e.g., \cite{CH17, CHLZ17}). 
Based on the idea that $\{\widetilde X_{n}\}_{n=1,\cdots,N}$ is the exponential Euler method applied to SPDE in \cite{BCH18,BG18}, we introduce the auxiliary problem as
\begin{align}\label{rspde}
dX^{\delta t}+AX^{\delta t}dt=\Psi_{\delta t}(X^{\delta t})dt
+dW(t),\quad 
X^{\delta t}(0)=X_0.
\end{align}

The differentiability of $\Psi_{\delta t}$ is listed on 
the following lemma, which generalizes the case in \cite[Lemma 2.1]{BG18}.

\begin{lm}\label{psi}
Under the conditions of Lemma \ref{ode-lm},
for $\delta t_0 \in (0,1]$, there exists $C(\delta t_0,f)>0$ such that for all $\delta t\in [0,\delta t_0]$ and $\xi\in \R$,
\begin{align*}
\Psi_{\delta t}'(\xi)&\le e^{C\delta t_0},\qquad \qquad\qquad\qquad\qquad\quad
|\Psi_{\delta t}'(\xi)|\le C(\delta t_0)(1+|\xi|^{K-1}),\\
|\Psi_{\delta t}''(\xi)|&\le C(\delta t_0)(1+|\xi|^{(2K-3)\lor (K-1)}),\quad 
|\Psi_{\delta t}(\xi)-\Psi_{0}(\xi)|\le C(\delta t_0)\delta t
(1+|\xi|^{2K-1}).
\end{align*}
\end{lm}
\begin{proof}
By the definition of $\Psi_{\delta t}$ and the properties of $\Phi_{\delta t}$ in Lemma \ref{ode-lm},
we have 
\begin{align*}
\Psi_{\delta t}'(\xi)&=\frac {\Phi_{\delta t}'(\xi)-1}{\delta t}=\frac {\int_0^{\delta t} f'(\Phi_s(\xi))\Phi_s'(\xi)ds}{\delta t}\le  C(f,\delta t_0),\\
|\Psi_{\delta t}'(\xi)|&\le \frac {|\int_0^{\delta t} f'(\Phi_s(\xi))\Phi_s'(\xi)ds|}{\delta t}
\le C(f, \delta t_0)\sup_{s\in[0,\delta t]}(1+|\Phi_s(\xi)|^{K-1})\\
&\le  C(f, \delta t_0)(1+|\xi|^{K-1}),\\
|\Psi_{\delta t}''(\xi)|&\le 
 \frac {|\int_0^{\delta t} f''(\Phi_s(\xi))(\Phi_s'(\xi))^2+f'(\Phi_s(\xi))\Phi_s''(\xi)ds|}{\delta t}\\
 &\le C(f,\delta t_0)(1+|\xi|^{(2K-3)\lor (K-1)}),\\ 
 |\Psi_{\delta t}(\xi)-\Psi_{0}(\xi)|&\le 
 \frac {|\int_0^{\delta t} \int_0^1 f'(\theta \Phi_{s}(\xi)+(1-\theta)\xi)(\Phi_{s}(\xi)-\xi) d\theta  ds|}{\delta t}\\
 &\le  \sup_{s\in[0,\delta t]}\int_0^1|f'(\theta \Phi_{s}(\xi)+(1-\theta)\xi)|d\theta  \sup_{s\in[0,\delta t]}|(\Phi_{s}(\xi)-\xi)| \\
 &\le C(f,\delta t_0)\delta t(1+|\xi|^{2K-1}).
 \end{align*}
\end{proof}

Based on Lemma \ref{psi}, the coefficient $\Psi_{\delta t}(\cdot)$ of  Eq. \eqref{rspde} is globally Lipschitz due to the fact that $\Psi_{t}(\xi)=\frac {\Phi_t(\xi)-\xi}t$. However, the Lipschitz coefficients of $\Psi_{t}, t\ge 0$ are not uniformly bounded with respect to $t$ (see, e.g., \cite{BG18b}). Indeed,  the solution of Eq. \eqref{rspde} is strongly convergent to that of Eq. \eqref{spde},
whose  proof is similar  to \cite[Proposition 4.8]{BG18}.

\begin{lm}\label{spl}
Let  Assumptions \ref{as-lap}-\ref{as-dri} hold.
Then the solution $X^{\delta t}$ of Eq. \eqref{rspde} is strongly convergent to the solution
$X$ of Eq. \eqref{spde} and satisfies, for any $p\ge 1$,
\begin{align*}
&\E\Big[\|X^{\delta t}(t)\|_E^p\Big]\le C(T,p)(1+\|X_0\|^p_{E}),\\
&\Big\|\sup_{t\in [0,T]} \|X^{\delta t}(t)-X(t)\|\Big\|_{L^p(\Omega)} \le C(X_0,T,p)\delta t.
\end{align*} 
\end{lm}

The idea of deducing the sharp weak convergence rate
lies on the decomposition of $\E\Big[\phi(X(t))-\phi(X^h(t))\Big]$ into $\E\Big[\phi(X(t))-\phi(X^{\delta t}(t))\Big]$ and $\E\Big[\phi(X^{\delta t}(t))-\phi(X^h(t))\Big]$. The first term is estimated by Lemma \ref{spl} and possesses the strong  convergence rate with respect to the parameter $\delta t$. 
The second error is estimated  by  utilizing  the regularity of Kolmogorov equation with respect to Eq. \eqref{rspde} and
integration by parts in the  sense of  Malliavin calculus.
Similar to \cite{BG18b, CJK14}, to get the rigorous regularity result   of  the Kolmogorov equation, the noise term $dW(t)$  is  
regularized as $e^{\delta A}dW(t)$, $\delta >0$. For convenience, we omit the procedure of regularizing the noise since the following 
proposition allows us to take the  limit $\delta \to 0$.

Next, we give the regularity of Kolmogorov equation with respect to Eq. \eqref{rspde}
\begin{align}\label{kol}
\frac {\partial U^{\delta t}(t,x)} {\partial t}
=\LL^{\delta t} U^{\delta t}(t,x)
&=\<-Ax+\Psi_{\delta t}(x), D U^{\delta t}(t,x)\>\\\nonumber
&\quad+\frac 12\text{tr}[Q^{\frac 12}D^2U^{\delta t}(t,x)Q^{\frac 12}].
\end{align}

\begin{prop}\label{reg-kol}
Let $\phi\in \mathcal C_b^2(\HH)$.
For every $\alpha, \beta, \gamma \in [0,1)$, $\beta+\gamma<1$, there exist $C(T, \delta t_0,\alpha)$ and 
$C(T, \delta t_0,\beta, \gamma )$ such that for $\delta t\in [0,\delta t_0]$, $x\in E, y,z \in \HH$ and $t\in (0,T]$,
\begin{align}\label{kol-d1}
|DU^{\delta t}(t,x).y|&\le \frac {C(T, \delta t_0,\alpha)(1+|x|_{E}^{K-1})|\phi|_{\mathcal C_b^1}}{t^{\alpha}}\|A^{-\alpha}y\|, \\\label{kol-d2}
|D^2U^{\delta t}(t,x).(y,z)|&\le \frac {C(T, \delta t_0,\beta, \gamma )(1+|x|_{E}^{(5K-6)\lor(4K-4)})|\phi|_{\mathcal C_b^2}}{t^{\beta+\gamma}}\|A^{-\beta}y\|\|A^{-\gamma }z\|.
\end{align} 
\end{prop}
\begin{proof}
Similar arguments in \cite[Theorem 4.1]{BG18b} prove  \eqref{kol-d1} and that
for $0\le \alpha<1$,
\begin{align}\label{var0}
\|\eta^y(t,x)\|\le \frac {C(T,\delta t_0,\alpha)}{t^{\alpha}}
\sup_{t\in [0,T]}\|\Psi'_{\delta t}(X^{\delta t}(t,x))\|_E\|A^{-\alpha}y\|,
\end{align}
where $\eta^{y}$ satisfies 
\begin{align*}
d\eta^y(t,x)&=(-A+\Psi_{\delta t}'(X^{\delta t}(t,x))\eta^y(t,x)dt, \; \eta^y(t,x)=y.
\end{align*}
Here we give a short  proof 
for \eqref{kol-d2} which is  different from the dual 
argument in \cite{BG18b}.
Notice that 
\begin{align*}
D^2U^{\delta t}(t,x).(y,z)&=
\E[D\phi(X^{\delta t}(t,x)).\zeta^{y,z}(t,x)]\\
&\quad+
\E[D^2\phi(X^{\delta t}(t,x)).(\eta^y(t,x), \eta^z(t,x))],
\end{align*}
where  $\zeta^{y,z}$ satisfies
\begin{align*}
d\zeta^{y,z}(t,x)&= (-A+\Psi_{\delta t}'(X^{\delta t}(t,x))\zeta^{y,z}(t,x)dt+\Psi''_{\delta t}(X^{\delta t}(t,x))\eta^y(t,x)\eta^z(t,x).
\end{align*}
Thus it suffices to prove the regularity of $\zeta^{y,z}$ thanks to \eqref{var0}.
Due to the fact that
\begin{align*}
\zeta^{y,z}(t,x)=\int_0^tV(t,s)\Big(\Psi''_{\delta t}(X^{\delta t}(s,x))\eta^y(s,x)\eta^z(s,x)\Big)ds,
\end{align*}
where
\begin{align*}
dV(t,s)z=(-A+\Psi'_{\delta t}(X^{\delta t}(t,x)))V(t,s)zdt, \quad V(s,s)z=z,
\end{align*}
we need to deduce more refined estimate of $V(t,s)z, \;0\le s<t\le T$.
The property of  $\Psi'_{\delta t}$ in Lemma \ref{psi}, combined with a energy estimate, yields that  $\|V(t,s)z\|^2\le C(T,\delta t_0) \|z\|^2$. 
Moreover, we claim that  for $0\le s<t \le T$, $0\le \alpha <1$,
\begin{align}\label{var}
\|V(t,s)y\|\le \frac{C(T,\delta t_0,\alpha) }{(t-s)^{\alpha}}\sup_{r\in[0,T]}\|\Psi'_{\delta t}(X^{\delta t}(r,x))\|_{E}\|A^{-\alpha }y\|.
\end{align}
Indeed, let $\widetilde V(t,s)y= V(t,s)y-e^{-(t-s)A}y, 0\le s\le t\le T$.
Then we have for $t>s$,
\begin{align*}
d\widetilde V(t,s)y&=(-A+\Psi'_{\delta t}(X^{\delta t}(t,x)))
\widetilde V(t,s)ydt+\Psi'_{\delta t}(X^{\delta t}(t,x))e^{-(t-s)A}ydt,\\
 \widetilde V(s,s)y&=0,
\end{align*}
and 
\begin{align*}
\|\widetilde V(t,s)y\|&\le 
\int_s^t\Big\|V(t,r)\Big(\Psi'_{\delta t }(X^{\delta t}(r,x))e^{-(r-s)A}y \Big)\Big\|dr\\
&\le C(T,\delta t_0)\sup_{r\in [0,T]} \|\Psi'_{\delta t }(X^{\delta t}(r,x))\|_E \int_s^t\|e^{-(r-s)A}y\|dr\\
&\le C(T,\delta t_0,\alpha)\sup_{r\in [0,T]} \|\Psi'_{\delta t }(X^{\delta t}(r,x))\|_E (t-s)^{1-\alpha}\|A^{-\alpha}y\|,
\end{align*}
which implies that the estimate \eqref{var} holds.
Now, we are in a position to prove \eqref{kol-d2}.
Based on \eqref{var} and Sobolev embedding theorem, we have for $\alpha >\frac 14$,
\begin{align*}
\|\zeta^{y,z}(t,x)\|&\le \int_0^t\|V(t,s)\Psi''_{\delta t}(X^{\delta t}(s,x))\eta^y(s,x)\eta^z(s,x)\|ds\\
&\le C(T,\delta t_0,\alpha )\int_0^t \sup_{r\in [0,T]} \|\Psi'_{\delta t }(X^{\delta t}(r,x))\|_E(t-s)^{-\alpha} \\
&\qquad\quad\Big\|A^{-\alpha}\Big(\Psi''_{\delta t}(X^{\delta t}(s,x))\eta^y(s,x)\eta^z(s,x)\Big)\Big\|ds\\
&\le C(T,\delta t_0,\alpha )\int_0^t \sup_{r\in [0,T]} \|\Psi'_{\delta t }(X^{\delta t}(r,x))\|_E(t-s)^{-\alpha}\|\Psi''_{\delta t}(X^{\delta t}(s,x))\|_E\\
&\qquad\|\eta^y(s,x)\|\|\eta^z(s,x)\|ds\\
&\le C(T,\delta t_0,\alpha)\sup_{r\in [0,T]}\|\Psi'_{\delta t }(X^{\delta t}(r,x))\|_E\sup_{r\in [0,T]}\|\Psi''_{\delta t}(X^{\delta t}(r,x))\|_E \\
&\qquad\int_0^t(t-s)^{-\alpha}\|\eta^y(s,x)\|\|\eta^z(s,x)\|ds.
\end{align*}
Now using the estimation  \eqref{var0}, the growth  of $\Psi_{\delta t}$, and stability of $X^{\delta t}$, we obtain
\begin{align*}
\E[\|\zeta^{y,z}(t,x)\|]
&\le C(T,\delta t_0,\alpha,\beta,\gamma)
\int_0^t(t-s)^{-\alpha}s^{-\beta-\gamma}ds
\|A^{-\beta}y\|\|A^{-\gamma}y\|\\
&\qquad\E\Big[\sup_{r\in [0,T]}\|\Psi'_{\delta t }(X^{\delta t}(r,x))\|_E^3\sup_{r\in [0,T]}\|\Psi''_{\delta t}(X^{\delta t}(r,x))\|_E\Big]\\
&\le C(T,\delta t_0,\alpha,\beta,\gamma)t^{-\beta-\gamma}(1+\|x\|_E^{(5K-6)\lor(4K-4)})
\|A^{-\beta}y\|\|A^{-\gamma}y\|,
\end{align*}
which completes the proof.
\end{proof}
\begin{rk}
The Sobolev embedding inequality 
$\|y\|_{L^{\infty}}\le C\|y\|_{\HH^{\frac d2+\epsilon}}$, $y\in \HH^{\frac d2+\epsilon}$, $\epsilon>0$, $d\le 3$,  yields that 
$\|A^{-\frac d2-\epsilon} y\|\le C\|y\|_{L^1}$.
Thus the regularity result of Kolmogorov equation in Proposition \ref{reg-kol} can be generalized to the higher dimensional case (d=2,3) and more regular noise case.
\end{rk}

\subsection{Weak convergence rate}
Before  studying  the weak convergence rate, 
we show that the numerical solution $X^h$ is differentiable in   Malliavin sense and  prove some estimates of $X^h$ needed later similar to \cite[Lemma 3.1]{AL16}. 

\begin{prop}\label{mal-dif}
Let  Assumptions \ref{as-lap}-\ref{as-dri} hold.
Then the Malliavin derivative of $X^h$ satisfies,  for some constant $C(T,\beta,X_0,K)$,
\begin{align*}
\E\Big[\|A_h^{\frac {\beta-1} 2}\mathcal D_sX^h(t)\|^2_{\LL_2^0}\Big]\le C(T,\beta,X_0,K)\big(1+\log(\frac 1h)\big)^{2(K-1)K^2}, \quad  0\le s\le t\le T.
\end{align*}
\end{prop}
\begin{proof}
Similar to the well-posedness of Eq. \eqref{sge}, we have that 
for $0\le s \le t\le T$, $y\in U_0$,
\begin{align*}
\mathcal D_s^yX^h(t)
&=S^h(t-s)P^hy+\int_s^tS^h(t-r)P^hDF(X^h(r))\cdot \mathcal D_s^yX^h(r)dr
\end{align*}
satisfies
\begin{align}\label{mal}
d \mathcal D_s^yX^h(t)
&=-A_h\mathcal D_s^yX^h(t)dt
+P^hDF(X^h(t))\cdot \mathcal D_s^yX^h(t)dt,\\\nonumber 
\mathcal D_s^yX^h(s)&=P^hy.
\end{align}
In order to get the estimate of $\|A_h^{\frac {\beta-1} 2}\mathcal D_sX^h(t)\|_{\LL_2^0}$, we first estimate 
$\|A_h^{-\gamma}\mathcal D_sX^h(t) y\|$, $0\le \gamma \le \frac 12$ and define  $\widetilde \eta_s(t,y) =\mathcal D_sX^h(t) y-S^h(t-s)P^hy$. 
Then $\widetilde \eta_s(t,y)$ satisfies the following equation
\begin{align*}
d\widetilde \eta_s(t,y)
&=-A_h\widetilde \eta_s(t,y)dt+P^h(DF(X^h(t))\cdot \widetilde \eta_s(t,y))dt\\
&\quad+P^h(DF(X^h(t))\cdot S^h(t-s)P^hy)dt,\\
\widetilde \eta_s(s,y)&=0,
\end{align*}
 and
 \begin{align*}
 \widetilde \eta_s(t,y)=\int_s^t\widehat V(t,r)P^h(DF(X^h(r))\cdot S^h(r-s)P^hy)dr,
 \end{align*}
where $\widehat V(t,r)z$ solves for $z\in V^h$,
\begin{align*}
d\widehat V(t,r)z=-A_h\widehat V(t,r)zdt+P^h(DF(X^h(t))\widehat V(t,r)z)dt, \quad \widehat V(r,r)z=z.
\end{align*}
The energy estimate, combined with the Gronwall's inequality, yields that for $s\le r\le t$,
\begin{align*}
\|\widehat V(t,r)z\|^2\le C(T)\|z\|^2.
\end{align*}
This implies that 
\begin{align*}
\|\widetilde \eta_s(t,y)\|&\le C(T)\int_s^t\|P^h(DF(X^h(r))\cdot S^h(r-s)P^hy)\|dr\\
&\le C(T,\gamma)\sup_{r\in [0,T]}\Big[1+\|X^h(r)\|_E^{K-1}\Big]\int_s^t (r-s)^{-\gamma}dr\|A^{-\gamma}y\|.
\end{align*}
Combining with the fact that for $ 0\le \gamma\le \frac 12$,
$$
\|S^h(t-s)P^hy\|\le C(T,\gamma)(t-s)^{-\gamma}\|A^{-\gamma}y\|,
$$
we get 
\begin{align*}
\|\mathcal D_s^yX^h(t) \|\le C(T,\gamma)\sup_{r\in [0,T]}\Big[1+\|X^h(r)\|_E^{K-1}\Big](t-s)^{-\gamma}
\|A^{-\gamma}y\|.
\end{align*}
Thus by the mild form of $\mathcal D_sX^h(t) y$ and the equivalence of norms in \eqref{eq-norm}, we obtain for $0\le \gamma \le \frac 12$,
\begin{align*}
\|A^{-\gamma}\mathcal D^y_sX^h(t)\|&\le 
\|A^{-\gamma}S^h(t-s)P^h y\|\\
&\quad+\int_s^t\|S^h(t-r)P^hDF(X^h(r))\cdot \mathcal D_s^yX^h(r)\|dr\\
&\le C(\gamma)\|A_h^{-\gamma}S^h(t-s)P^h y\|\\
&\quad
+C(T,\gamma)\sup_{r\in [0,T]}[1+\|X^h(r)\|^{K-1}_{E}]\int_s^t \| \mathcal D_s^yX^h(r)\|dr\\
&\le C(T,\gamma)\|A^{-\gamma} y\|
\Big(1+\sup_{r\in [0,T]}[1+\|X^h(r)\|^{2K-2}_{E}]
\int_s^t (r-s)^{-\gamma}ds\Big)\\
&\le C(T,\gamma)\Big(1+\sup_{r\in [0,T]}\|X^h(r)\|^{2K-2}_{E}\Big)\|A^{-\gamma} y\|.
\end{align*}
Now, taking $-\gamma=\frac {\beta -1}2$,  $0<\beta\le 1$, $y=Q^{\frac 12}e_i, i\in \N^+$, together with the stability result of $X^h$ in Proposition \ref{pri-fem}, yields that
\begin{align*}
\E\Big[\Big\|A^{\frac {\beta -1}2}\mathcal D_sX^h(t)\Big\|_{\LL_2^0}^2\Big]
&\le C(T,\beta)\sum_{i\in \N^+}\E\Big[\Big(1+\sup_{r\in [0,T]}\|X^h(r)\|^{4(K-1)}_{E}\Big)\|A^{\frac {\beta-1}2}Q^{\frac 12}e_i\|^2\Big]\\
&\le C(T,X_0,\beta)\big(1+(\log(\frac 1h))^{2(K-1)K^2}\big),
\end{align*}
which completes the proof.
\end{proof}

Now, we turn to estimate the weak error of $\Big|\E\Big[\phi(X^{\delta t}(t))-\phi(X^h(t))\Big]\Big|$.

\begin{tm}\label{weak}
Let  Assumptions \ref{as-lap}-\ref{as-dri} hold.  Assume in addition that $|f''(\xi)|
\le L_f(1+|\xi|^{K-2}), \; 2\le K< 5$, then for every test functions $\phi\in \mathcal C_b^2(\HH)$, $T>0$, $\beta \in (0,1]$ and   $\gamma<\beta$,
there exists $C(X_0,T,\beta,\phi)$ such that 
\begin{align*}
\Big|\E\Big[\phi(X^{\delta t}(T))-\phi(X^h(T))\Big]\Big|
\le C(X_0,T,\beta,\phi)\Big(h^{2\gamma}+\delta t(\log(\frac 1h))^{\frac {(3K-2)K^2}2}\Big).
\end{align*}
\end{tm}

\begin{proof}
Based on the property $\E\Big[\phi(X^h(T))\Big]=\E\Big[U^{\delta t}(0,X^h(T))\Big]$,
we split the weak error as
\begin{align*}
\Big|\E\Big[\phi(X^{\delta t}(T))-\phi(X^h(T))\Big]\Big|
&\leq\Big|\E\Big[\phi(X^{\delta t}(T))\Big]-\E\Big[U^{\delta t}(T,X^h(0))\Big]\Big|\\
&\quad+\Big|\E\Big[U^{\delta t}(T,X^h(0))\Big]-\E\Big[\phi(X^h(T))\Big]\Big|.
\end{align*}
By the regularity of $DU^{\delta t}(t,x)$ \eqref{kol-d1} in Proposition \ref{reg-kol}, we bound the first error as  for $0\le \alpha<1$,
\begin{align*}
&\Big|\E\Big[\phi(X^{\delta t}(T))\Big]-\E\Big[U^{\delta t}(T,X^h(0))\Big]\Big|\\
&=
\Big|U^{\delta t}(T,X(0))-U^{\delta t}(T,X^h(0))\Big|\\
&= C(T,\delta t_0) \Big|\int_0^1DU^{\delta t}(T,\theta X(0)+(1-\theta)X^h(0)) d\theta \cdot ((I-P^h)X(0))\Big|\\
&\le C(T,\delta t_0,\alpha,\phi)T^{-\alpha}  \E\Big[\Big(1+\|X_0\|_E^{K-1}+\|X^h(0)\|_E^{K-1}\Big) \Big]\|(-A)^{-\alpha}(I-P^h)X(0)\|\\
&\le C(T,\delta t_0,\alpha,\phi,X_0)T^{-\alpha}h^{2\alpha},
\end{align*}
where we use the fact $\|A^{-\alpha}(I-P^h)\|_{\LL(\HH)}=\|A^{-\alpha}(I-P^h))^*\|_{\LL(\HH)}=\|(I-P^h)A^{-\alpha}\|_{\LL(\HH)}$ and 
the estimation \eqref{err-ord}.

Next, we aim to estimate the left error $\Big|\E\Big[U^{\delta t}(T,X^h(0))\Big]-\E\Big[\phi(X^h(T))\Big]\Big|$.
We recall the Markov generator $\mathcal L^h$ of $X^h$,
\begin{align*}
(\mathcal L^h U)(x)&=\<-A_hx+P^hF(x),DU(x)\>
+\frac 12 \text{tr}[P^hQP^hD^2U(x)],
\end{align*}
where $U\in \mathcal C^2(\HH,\R), x\in V_h.$
Then It\^o formula and corresponding Kolmogorov equation  \eqref{kol} yield that 
\begin{align*}
&\E\Big[U^{\delta t}(T,X^h(0))\Big]-\E\Big[\phi(X^h(T))\Big]\\
&= \E\Big[U^{\delta t}(T,X^h(0))-U^{\delta t}(0,X^h(T))\Big]\\
&=-\E\Big[\int_0^T\Big(-\dot U^{\delta t}(T-t,X^h(t))+\mathcal L^hU^{\delta t}(T-t, X^h(t))\Big)dt\Big]\\
&=\E\Big[\int_0^T\Big(\big(\mathcal L^{\delta t}-\mathcal L^h\big)U^{\delta t}(T-t, X^h(t))\Big)dt\Big].
\end{align*}
Based on the expressions of $\mathcal L^{\delta t}$ and
$\mathcal L^h$, we obtain 
\begin{align*}
&\Big|\E\Big[U^{\delta t}(T,X^h(0))\Big]-\E\Big[\phi(X^h(T))\Big]\Big|\\
&\le \Big|\E\Big[ \int_0^T\Big\< (A-A_h)X^h(t), DU^{\delta t}(T-t,X^h(t)) \Big\>dt  \Big]\Big|\\
&\quad+ \Big|\E\Big[\int_0^T\Big\<\Psi_{\delta t}(X^h(t))-P^hF(X^h(t)), DU^{\delta t}(T-t,X^h(t))\Big\>dt   \Big]\Big|\\
&\quad+ \frac 12\Big|\E\Big[ \int_0^T\text{tr}\Big\{QD^2U^{\delta t}(T-t,X^h(t))-P^hQP^hD^2U^{\delta t}(T-t,X^h(t))\Big\}dt  \Big]\Big|\\
&=: e^1(T)+e^2(T)+e^3(T).
\end{align*}
The relation $R^h=A_h^{-1}P^hA$ implies that
\begin{align*}
&\<(A-A_h)X^h(t),DU^{\delta t}(T-t,X^h(t))\>\\
&=\<(AP^h-P^hA_h)X^h(t),DU^{\delta t}(T-t,X^h(t))\>\\
&=\<X^h,(P^hA-A_hP^h)DU^{\delta t}(T-t,X^h(t))\>\\
&=\<X^h,A_hP^h(A_h^{-1}P^hA-I)DU^{\delta t}(T-t,X^h(t))\>\\
&=\<X^h,A_hP^h(R^h-I)DU^{\delta t}(T-t,X^h(t))\>.
\end{align*}
The above equality and the mild form of $X^h$ lead to
\begin{align*}
&e^1(T)\\
&\le \Big|\E\Big[ \int_0^T\Big\< S^h(t)X^h(0), A_hP^h(R^h-I)DU^{\delta t}(T-t,X^h(t)) \Big\>dt  \Big]\Big|\\
&+\Big|\E\Big[ \int_0^T\Big\< \int_0^tS^h(t-s)P^hF(X^h(s))ds, A_hP^h(R^h-I)DU^{\delta t}(T-t,X^h(t)) \Big\>dt  \Big]\Big|\\
&+\Big|\E\Big[ \int_0^T\Big\< \int_0^tS^h(t-s)P^hdW(s), A_hP^h(R^h-I)DU^{\delta t}(T-t,X^h(t)) \Big\>dt  \Big]\Big|\\
&=:e^{1,1}(T)+e^{1,2}(T)+e^{1,3}(T)
\end{align*}
Applying the regularity estimate of $DU^{\delta t}$ \eqref{kol-d1}, 
the smoothing property of $S^h$ \eqref{smo-ah} and the stability of $X^h$ in Proposition \eqref{pri-fem}, it follows that for small $\epsilon >0$, $\epsilon<\alpha<1$,
\begin{align*}
e^{1,1}(T)
&=
\Big|\E\Big[ \int_0^T\Big\< A_h^{1-\epsilon}S^h(t)X^h(0), A_h^{\epsilon}P^h(R^h-I)A^{-\alpha}
\\
&\qquad \quad A^{\alpha }DU^{\delta t}(T-t,X^h(t)) \Big\>dt  \Big]\Big|\\
&\le 
\E\Big[\int_0^T\Big\|A_h^{1-\epsilon}S^h(t)X^h(0)\Big\| 
\Big\|A_h^{\epsilon}P^h(R^h-I)A^{-\alpha}\Big\|_{\LL(\HH)}\\
&\qquad\quad \Big\|A^{\alpha }DU^{\delta t}(T-t,X^h(t))\Big\| dt\Big]\\
&\le C(T,\epsilon)h^{2\alpha-2\epsilon}\int_0^T
t^{-1+\epsilon}\|X_0\|\E\Big[\Big\|A^{\alpha }DU^{\delta t}(T-t,X^h(t))\Big\|\Big]dt\\
&\le  C(T,\epsilon,\alpha,\phi)h^{2\alpha-2\epsilon}\int_0^T
t^{-1+\epsilon}(T-t)^{-\alpha}\|X_0\|\sup_{t\in[0,T]}\E\Big[1+\Big\|X^h(t)\Big\|_E^{K-1}\Big]dt.
\end{align*}
Similar arguments yield that 
\begin{align*}
e^{1,2}(T)
&=\Big|\E\Big[ \int_0^T\Big\< \int_0^t A_h^{1-\epsilon}S^h(t-s)P^hF(X^h(s))ds, A_h^{\epsilon }P^h(R^h-I)A^{-\alpha }\\
&\qquad\quad A^{\alpha }
DU^{\delta t}(T-t,X^h(t)) \Big\>dt  \Big]\Big|\\
&\le \E\Big[ \int_0^T\int_0^t \Big\|A_h^{1-\epsilon}S^h(t-s)P^hF(X^h(s))\Big\|ds\\
&\quad \qquad\Big\|A_h^{\epsilon }P^h(R^h-I)A^{-\alpha } \Big\|_{\LL(\HH)} \Big\|A^{\alpha}
DU^{\delta t}(T-t,X^h(t))\Big\|dt
 \Big]\\
 &\le C(T,\epsilon,\alpha)h^{2\alpha-2\epsilon} \sup_{t\in [0,T]}\sqrt{\E\Big[\|F(X^h(t))\|^2\Big]}\\
 &\quad\int_0^T\int_0^t \sqrt{\E\Big[\|A^{\alpha}DU^{\delta t}(T-t,X^h(t))\|^2\Big]}(t-s)^{-1+\epsilon}dsdt
\\
 &\le C(T,\epsilon,\alpha,\phi)h^{2\alpha-2\epsilon}
 \sup_{t\in [0,T]}\sqrt{\E\Big[1+\|X^h(t)\|_{E}^{2K}\Big]}
\\
 &\quad \sup_{t\in [0,T]}\sqrt{\E\Big[1+\|X^h(t))\|_{E}^{2K-2}\Big]}\int_0^T\int_0^t (T-t)^{-\alpha}(t-s)^{-1+\epsilon}dsdt.
\end{align*}
To deal with $e^{1,3}(T)$, we make use of  the integration by parts formula in Malliavin sense \eqref{int-by} and the chain rule  to get 
\begin{align*}
e^{1,3}(T)&=\Big|\E\Big[ \int_0^T\Big\< \int_0^tS^h(t-s)P^hdW(s), A_hP^h(R^h-I)DU^{\delta t}(T-t,X^h(t)) \Big\>dt  \Big]\Big|
\\
&=\Big|\E\Big[ \int_0^T \int_0^t\Big\<S^h(t-s)P^h, \\
&\qquad \qquad A_hP^h(R^h-I)D^2U^{\delta t}(T-t,X^h(t)) \mathcal D_sX^h(t)\Big\>_{\LL_2^0}dsdt  \Big]\Big|.
\end{align*}
Then by the property of Hilbert--Schmidt operator and
Cauchy--Schwarz inequality, we have 
\begin{align*}
&e^{1,3}(T)\\
&=\Big|\E\Big[ \int_0^T \int_0^t\Big\<A_h^{\frac {1+\beta } 2-\epsilon}S^h(t-s)A_h^{\frac {1-\beta} 2 }A_h^{\frac {\beta -1}2}P^h, \\
&\quad \quad A_h^{\frac {1-\beta}2+\epsilon}P^h(R^h-I)A^{-\frac {1+\beta }2+\epsilon} A^{\frac {1+\beta}2-\epsilon} D^2U^{\delta t}(T-t,X^h(t)) \mathcal D_sX^h(t)\Big\>_{\LL_2^0}dsdt  \Big]\Big|\\
&\le\E\Big[ \int_0^T \int_0^t \Big\|A_h^{1-\epsilon}S^h(t-s)P^h\Big\|_{\LL(\HH)}\|A_h^{\frac {\beta -1}2}P^h\|_{\LL_2^0}\Big\|A_h^{\frac {1-\beta}{2}+\epsilon}P^h(R^h-I)\\
&\quad\quad
A^{-\frac {\beta+1}2+\epsilon} \Big\|_{\LL(\HH)}\Big\|A^{\frac {1+\beta}2-\epsilon} D^2U^{\delta t}(T-t,X^h(t))A^{\frac {1-\beta}2}\Big\|_{\LL(\HH)}
\Big\|A^{\frac {\beta-1}2}\mathcal D_sX^h(t)\Big\|_{\LL_2^0}dsdt\Big].
\end{align*}
Combining the regularity result of $DU^{\delta t}$  \eqref{kol-d1} and $D^2U^{\delta t}$  \eqref{kol-d2}, 
the smoothing property of $S^h$, $\|A_h^{\frac {\beta -1}2}P^h\|_{\LL_2^0}\le C\|A^{\frac {\beta -1}2}\|_{\LL_2^0}$ and the stability of $X^h$ together, we obtain 
\begin{align*}
e^{1,3}(T)&\le C(T,\epsilon)h^{2\beta-4\epsilon}\int_0^T \int_0^t (t-s)^{-1+\epsilon}\sqrt{\E\Big[\Big\|A^{\frac {\beta-1}2}\mathcal D_sX^h(t)\Big\|_{\LL_2^0}^2\Big]}\\
&\qquad\qquad\qquad\qquad\qquad
\sqrt{\E\Big[\Big\|A^{\frac {1+\beta}2-\epsilon} D^2U^{\delta t}(T-t,X^h(t))A^{\frac {1-\beta}2}\Big\|_{\LL(\HH)}
^2\Big]} 
dsdt\\
&\le C(T,\epsilon)h^{2\beta-4\epsilon}\int_0^T \int_0^t (t-s)^{-1+\epsilon}
(T-t)^{-1+\epsilon}\sqrt{\E\Big[\Big\|A^{\frac {\beta-1}2}\mathcal D_sX^h(t)\Big\|_{\LL_2^0}^2\Big]}
\\
&\qquad\qquad\qquad\qquad\qquad
\sup_{t\in [0,T]}\sqrt{\E\Big[1+\|X^h(t)\|_{E}^{10K-12}\Big]} 
dsdt.
\end{align*} 
Proposition  \ref{mal-dif} yields that 
\begin{align*}
e^{1,3}(T)&\le 
C(X_0,T,\epsilon,\beta)h^{2\beta-4\epsilon}\Big(1+\log(\frac 1h)\Big)^{\frac {(7K-8)K^2}2}.
\end{align*}
Thus we conclude  that $e^1(T)\le C(X_0,T,\epsilon,\beta)(1+T^{-\beta }+(\log(\frac 1h))^{\frac {(7K-8)K^2}2})h^{2\beta-4\epsilon}.$

Next, we turn to focus on $e^{2}(T)$.
From $\Psi_0=F$, it follows that
\begin{align*}
e^{2}(T)&\le \Big|\E\Big[\int_0^T\Big\<\Psi_{\delta t}(X^h(t))-\Psi_0(X^h(t)), DU^{\delta t}(T-t,X^h(t))\Big\>dt   \Big]\Big|\\
&\quad+\Big|\E\Big[\int_0^T\Big\<(I-P^h)F(X^h(t)), DU^{\delta t}(T-t,X^h(t))\Big\>dt   \Big]\Big|\\
&=:e^{2,1}(T)+e^{2,2}(T).
\end{align*}
By the continuity of $\Psi_t$ with respect to $t$ in Lemma 
\ref{psi} and
the regularity of $DU^{\delta t}$\eqref{kol-d1}, it leads to
\begin{align*}
e^{2,1}(T)
&\le\E\Big[\int_0^T \Big|\Big\<\Psi_{\delta t}(X^h(t))-\Psi_0(X^h(t)), DU^{\delta t}(T-t,X^h(t))\Big\>\Big|dt   \Big]\\
&\le\E\Big[\int_0^T\Big\|DU^{\delta t}(T-t,X^h(t))\Big\|\Big\|\Psi_{\delta t}(X^h(t))-\Psi_0(X^h(t))\Big\|dt\Big]\\
&\le C(T,\delta t_0)
\delta t\Big(1+\sup_{t\in [0,T]}\|X^h(t)\|^{3K-2}_{E}\Big).
\end{align*}
The regularity of $DU^{\delta t}$ \eqref{kol-d1}, the estimate \eqref{err-ord}, the growth of $F$
and the stability of $X^h$ yield that
\begin{align*}
e^{2,2}(T)&\le \E\Big[\int_0^T\Big\|(I-P^h)A^{-1 +\epsilon }\Big\|_{\LL(\HH)} \Big\|A^{1-\epsilon}DU^{\delta t}(T-t,X^h(t))\Big\|_{\LL(\HH)}
\\
&\qquad\qquad\|F(X^h(t))\|dt   \Big]\\
&\le C(T,\epsilon)h^{2-2\epsilon} \int_0^T(T-t)^{-1+\epsilon}\sqrt{\E\Big[\|F(X^h(t))\|^2\Big]}\sqrt{\E\Big[1+\|X^h(t)\|_E^{2K-2}\Big]}dt\\
&\le C(T,\epsilon)h^{2-2\epsilon}\sup_{t\in[0,T]} \E\Big[1+\|X^h(t)\|_E^{2K-1}\Big].
\end{align*}
Summing up the estimations of $e^{2,1}$ and $e^{2,2}$, we deduce that 
\begin{align*}
e^2(T)\le C(X_0,T,\epsilon)\Big(h^{2-2\epsilon}(1+(\log(\frac 1h))^{\frac {(2K-1)K^2}2})+\delta t(1+(\log(\frac 1h))^{\frac {(3K-2)K^2}2})\Big).
\end{align*}

For the last term $e^3(T)$, we have 
\begin{align*}
e^3(T)&=\frac 12\Big|\E\Big[ \int_0^T\text{tr}\Big\{\big(IQ(I-P^h)+(I-P^h)QP^h\big)D^2U^{\delta t}(T-t,X^h(t))\Big\}dt  \Big]\Big|\\
&\le \frac 12\Big|\E\Big[ \int_0^T\text{tr}\Big\{IQ(I-P^h)D^2U^{\delta t}(T-t,X^h(t))\Big\}dt  \Big]\Big|\\
&\quad+\frac 12\Big|\E\Big[ \int_0^T\text{tr}\Big\{(I-P^h)QP^hD^2U^{\delta t}(T-t,X^h(t))\Big\}dt  \Big]\Big|\\
&=:e^{3,1}(T)+e^{3,2}(T).
\end{align*}
The properties of trace and  Hilbert--Schmidt operator lead to
\begin{align*}
e^{3,1}(T)&=
\frac 12\Big|\E\Big[ \int_0^T\text{tr}\Big\{IQ(I-P^h)D^2U^{\delta t}(T-t,X^h(t))A^{\frac {1-\beta}2} A^{\frac {\beta-1}2}\Big\}dt  \Big]\Big|\\
&=\frac 12\Big|\E\Big[ \int_0^T\text{tr}\Big\{A^{\frac {\beta-1}2}Q(I-P^h)D^2U^{\delta t}(T-t,X^h(t))A^{\frac {1-\beta}2} \Big\}dt  \Big]\Big|\\
&=\frac 12\Big|\E\Big[ \int_0^T\text{tr}\Big\{A^{\frac {\beta -1}2}QA^{\frac {\beta -1}2}A^{\frac {1-\beta}2}(I-P^h)A^{-\frac {1+\beta}2+\epsilon} \\
&\qquad\qquad A^{\frac {1+\beta}2-\epsilon}D^2U^{\delta t}(T-t,X^h(t))A^{\frac {1-\beta}2} \Big\}dt  \Big]\Big|\\
&\le \frac 12\int_0^T\E\Big[ \|A^{\frac {\beta-1}2}\|_{\LL_2^0}^2\Big\|A^{\frac {1-\beta}2}(I-P^h)A^{-\frac {1+\beta}2+\epsilon}\Big\|_{\LL(\HH)} \\
&\qquad\qquad \Big\|A^{\frac {1+\beta}2-\epsilon}D^2U^{\delta t}(T-t,X^h(t))A^{\frac {1-\beta}2} \Big\|_{\LL(\HH)} \Big]dt.
\end{align*}
Then the regularity of $D^2U^{\delta t}$ \eqref{kol-d2}, the estimate \eqref{err-ord}  and Proposition \ref{pri-fem} yield that 
\begin{align*}
e^{3,1}(T)&\le C(T,\beta,\epsilon)\|A^{\frac {\beta-1}2}\|_{\LL_2^0}^2h^{2\beta-2\epsilon} \int_0^T(T-t)^{-1+\epsilon}\sup_{t\in [0,T]}\E\Big[1+\|X^h(t)\|_E^{5K-6}\Big]dt\\
&\le C(T,X_0,\beta,\epsilon)h^{2\beta-2\epsilon}\Big(1+(\log(\frac 1h))^{\frac {(5K-6)K^2} 2}\Big).
\end{align*}
Similarly, we have  
\begin{align*}
e^{3,2}(T)&=\frac 12\Big|\E\Big[ \int_0^T\text{tr}\Big\{A^{-\frac {\beta+1}2+\epsilon } 
(I-P^h)A^{\frac {1-\beta}2 }A^{\frac {\beta-1}2} QA^{\frac {\beta -1}2}\\
&\qquad A^{\frac {1-\beta}2}P^hA^{\frac {\beta-1}2}A^{\frac {1-\beta}2}D^2U^{\delta t}(T-t,X^h(t))A^{\frac {1+\beta }2-\epsilon}\Big\}dt  \Big]\Big|\\
&\le C\E\Big[ \int_0^T\|A^{\frac {\beta-1}2}\|_{\LL_2^0}^2 \Big\| A^{\frac {1-\beta}2}
(I-P^h)A^{-\frac {1+\beta}2+\epsilon }\Big\|_{\LL(\HH)}
\Big\|A^{\frac {1-\beta}2}P^h\\
&\qquad A^{\frac {\beta-1}2}\Big\|_{\LL(\HH)}\Big\|A^{\frac {1-\beta}2}D^2U^{\delta t}(T-t,X^h(t))A^{\frac {1+\beta }2-\epsilon}\Big\|_{\LL(\HH)}dt  \Big]\\
&\le C(T,X_0,\beta,\epsilon)h^{2\beta-2\epsilon}
\int_0^T\Big\|A^{\frac {1-\beta}2 }P^hA^{\frac {\beta-1}2}\Big\|_{\LL(\HH)}(T-t)^{-1+\epsilon}dt\\
&\le C(T,X_0,\beta,\epsilon)h^{2\beta-2\epsilon}\Big(1+(\log(\frac 1h))^{\frac {(5K-6)K^2}2}\Big),
\end{align*}
where we use the property $\Big\|A^{\frac {\beta-1}2}P^hA^{\frac {1-\beta}2}\Big\|_{\LL(\HH)}\le C$ proven by the equivalence of norms \eqref{eq-norm},
\begin{align*}
\Big\|A^{\frac {\beta-1}2 }P^hA^{\frac {1-\beta}2}\Big\|_{\LL(\HH)}
&\le C\Big\|A_h^{\frac {\beta-1}2 }P^hA^{\frac {1-\beta}2}\Big\|_{\LL(\HH)}\le  C\Big\|A^{\frac {\beta-1}2}A^{\frac {1-\beta}2}\Big\|_{\LL(\HH)}\le C.
\end{align*}
The estimations of $e^{3,1}(T)$ and $e^{3,2}(T)$ indicate 
\begin{align*}
e^3(T)\le C(T,X_0,\beta,\epsilon)h^{2\beta-2\epsilon}\Big(1+(\log(\frac 1h))^{\frac {(5K-6)K^2}2}\Big).
\end{align*}
Summing up all the estimations of $\Big|\E\Big[\phi(X^{\delta t}(T))\Big]-\E\Big[U^{\delta t}(T,X^h(0))\Big]\Big|$, $e^1(T)$, $e^2(T)$ and $e^3(T)$, we obtain that for any small $\epsilon_1>\epsilon$,
\begin{align*}
&\Big|\E\Big[\phi(X^{\delta t}(T))-\phi(X^h(T))\Big]\Big|\\
&\le  C(X_0,T,\epsilon,\beta)h^{2\beta-4\epsilon}\big(1+T^{-\beta}+(\log(\frac 1h))^{\frac {(7K-8)K^2}2}\big)\\
&\quad+C(X_0,T,\epsilon)\delta t\big(1+(\log(\frac 1h))^{\frac {(3K-2)K^2}2}\big)\\
&\le C(X_0,T,\beta,\epsilon)\Big(h^{2\beta-4\epsilon_1}+\delta t(\log(\frac 1h))^{\frac {(3K-2)K^2}2}\Big),
\end{align*}
which, combined with a standard argument, finishes
the proof. 
\end{proof}

Combining Lemma \ref{spl} with Theorem \ref{weak}, we deduce the essentially sharp weak convergence rate
of the finite element method approximating Eq. \eqref{spde}. The essentially sharp weak convergence rate 
is in the sense that the weak convergence rate is essentially  twice the strong convergence rate.
We remark that even if the  logarithmic factor can be eliminated when $\beta >\frac 12$ or $Q=I$, the weak convergence rate can not be improved (see e.g. \cite[Theorem 1.1]{AL16}).

\begin{tm}
Let $T>0$.
Under the assumptions of Theorem \ref{weak}, 
for $\phi\in\mathcal C_b^2(\HH)$, $\beta\in[0,1)$, $\gamma <\beta$,
there exists  $C(X_0,T,\beta,\phi)>0$
such that
\begin{align*}
\Big|\E\Big[\phi(X(T))-\phi(X^h(T))\Big]\Big|
\le C(X_0,T,\beta,\phi)h^{2\gamma}.
\end{align*}
\end{tm}

\begin{proof}
By the triangle inequality, Lemma \ref{spl} and Theorem \ref{weak} and taking $\delta t=\mathcal O(h^{2\beta})$, we have 
\begin{align*}
\Big|\E\Big[\phi(X(T))-\phi(X^h(T))\Big]\Big|
&\le \Big|\E\Big[\phi(X(T))-\phi(X^{\delta t}(T))\Big]\Big|\\
&\quad +\Big|\E\Big[\phi(X^{\delta t}(T))-\phi(X^h(T))\Big]\Big|\\
&\le  C(X_0,T,p)\delta t\\
&\quad+C(X_0,T,\beta)\Big(h^{2\gamma}+\delta t(\log(\frac 1h))^{\frac {(3K-2)K^2}2}\Big)\\
&\le
 C(X_0,T,p)h^{2\gamma},
\end{align*}
which completes the proof.
\end{proof}

\section*{Acknowledgments}
The authors would like to thank Charles-Edouard Br\'ehier for his comments based on carefully reading the manuscript.

\bibliographystyle{amsplain}
\bibliography{bib}

\providecommand{\bysame}{\leavevmode\hbox to3em{\hrulefill}\thinspace}
\providecommand{\MR}{\relax\ifhmode\unskip\space\fi MR }
\providecommand{\MRhref}[2]{%
  \href{http://www.ams.org/mathscinet-getitem?mr=#1}{#2}
}
\providecommand{\href}[2]{#2}
\begin{thebibliography}{10}

\bibitem{AKL16}
A.~Andersson, R.~Kruse, and S.~Larsson, \emph{Duality in refined
  {S}obolev-{M}alliavin spaces and weak approximation of {SPDE}}, Stoch.
  Partial Differ. Equ. Anal. Comput. \textbf{4} (2016), no.~1, 113--149.
  \MR{3274889}

\bibitem{AL16}
A.~Andersson and S.~Larsson, \emph{Weak convergence for a spatial approximation
  of the nonlinear stochastic heat equation}, Math. Comp. \textbf{85} (2016),
  no.~299, 1335--1358. \MR{3454367}

\bibitem{AC17}
R.~Anton, D.~Cohen, and L.~Quer-Sardanyons, \emph{A fully discrete
  approximation of the one-dimensional stochastic heat equation}, IMA Journal
  of Numerical Analysis, dry060, https://doi.org/10.1093/imanum/dry060 (2018).

\bibitem{BGJK17}
S.~Becker, B.~Gess, A.~Jentzen, and P.~E. Kloeden, \emph{Strong convergence
  rates for explicit space-time discrete numerical approximations of stochastic
  {A}llen--{C}ahn equations}, arXiv:1711.02423 (2017).

\bibitem{BJ16}
S.~Becker and A.~Jentzen, \emph{Strong convergence rates for
  nonlinearity-truncated {E}uler-type approximations of stochastic
  {G}inzburg-{L}andau equations}, Stochastic Process. Appl. \textbf{129}
  (2019), no.~1, 28--69. \MR{3906990}

\bibitem{BCH18}
C.~E. Br\'ehier, J.~Cui, and J.~Hong, \emph{Strong convergence rates of
  semi-discrete splitting approximations for stochastic {A}llen--{C}ahn
  equation}, IMA J. Numer. Anal., dry052, https://doi.org/10.1093/imanum/dry052
  (2018).

\bibitem{bd18}
C.~E. Br\'{e}hier and A.~Debussche, \emph{Kolmogorov equations and weak order
  analysis for {SPDE}s with nonlinear diffusion coefficient}, J. Math. Pures
  Appl. (9) \textbf{119} (2018), 193--254. \MR{3862147}

\bibitem{BG18}
C.~E. Br\'ehier and L.~Gouden\`ege, \emph{Analysis of some splitting schemes
  for the stochastic {A}llen--{C}ahn equation}, arXiv:1801.06455, to appear in
  Discrete Contin. Dyn. Syst. Ser. B (2018).

\bibitem{BG18b}
\bysame, \emph{Weak convergence rates of splitting schemes for the stochastic
  {A}llen--{C}ahn equation}, arXiv:1804.04061 (2018).

\bibitem{CJK14}
D.~Conus, A.~Jentzen, and R.~Kurniawan, \emph{Weak convergence rates of
  spectral {G}alerkin approximations for {SPDE}s with nonlinear diffusion
  coefficients}, Ann. Appl. Probab. \textbf{29} (2019), no.~2, 653--716.
  \MR{3910014}

\bibitem{CLT94}
M.~Crouzeix, S.~Larsson, and V.~Thom\'ee, \emph{Resolvent estimates for
  elliptic finite element operators in one dimension}, Math. Comp. \textbf{63}
  (1994), no.~207, 121--140. \MR{1242058}

\bibitem{CH17}
J.~Cui and J.~Hong, \emph{Analysis of a splitting scheme for damped stochastic
  nonlinear {S}chr\"{o}dinger equation with multiplicative noise}, SIAM J.
  Numer. Anal. \textbf{56} (2018), no.~4, 2045--2069. \MR{3826675}

\bibitem{CHL16b}
J.~Cui, J.~Hong, and Z.~Liu, \emph{Strong convergence rate of finite difference
  approximations for stochastic cubic {S}chr\"odinger equations}, J.
  Differential Equations \textbf{263} (2017), no.~7, 3687--3713. \MR{3670034}

\bibitem{CHLZ17}
J.~Cui, J.~Hong, Z.~Liu, and W.~Zhou, \emph{Strong convergence rate of
  splitting schemes for stochastic nonlinear {S}chr\"{o}dinger equations}, J.
  Differential Equations \textbf{266} (2019), no.~9, 5625--5663. \MR{3912762}

\bibitem{CHS18a}
J.~Cui, J.~Hong, and L.~Sun, \emph{Strong convergence rate of a full
  discretization for stochastic {C}ahn--{H}illiard equation driven by
  space-time white noise}, arXiv:1812.06289 (2018).

\bibitem{CHS18}
\bysame, \emph{Weak convergence and invariant measure of a full discretization
  for parabolic {SPDE}s with non-globally lipschitz coefficients},
  arXiv:1811.04075 (2018).

\bibitem{DZ14}
G.~Da~Prato and J.~Zabczyk, \emph{Stochastic equations in infinite dimensions},
  second ed., Encyclopedia of Mathematics and its Applications, vol. 152,
  Cambridge University Press, Cambridge, 2014. \MR{3236753}

\bibitem{Deb11}
A.~Debussche, \emph{Weak approximation of stochastic partial differential
  equations: the nonlinear case}, Math. Comp. \textbf{80} (2011), no.~273,
  89--117. \MR{2728973}

\bibitem{FLZ17}
X.~Feng, Y.~Li, and Y.~Zhang, \emph{Finite element methods for the stochastic
  {A}llen--{C}ahn equation with gradient-type multiplicative noise}, SIAM J.
  Numer. Anal. \textbf{55} (2017), no.~1, 194--216. \MR{3600370}

\bibitem{HJK16}
M.~Hefter, A.~Jentzen, and R.~Kurniawan, \emph{Weak convergence rates for
  numerical approximations of stochastic partial differential equations with
  nonlinear diffusion coefficients in {UMD} banach spaces.}, arXiv:1612.03209
  (2016).

\bibitem{KLL12}
M.~Kov\'{a}cs, S.~Larsson, and F.~Lindgren, \emph{Weak convergence of finite
  element approximations of linear stochastic evolution equations with additive
  noise}, BIT \textbf{52} (2012), no.~1, 85--108. \MR{2891655}

\bibitem{KLL13}
\bysame, \emph{Weak convergence of finite element approximations of linear
  stochastic evolution equations with additive noise {II}. {F}ully discrete
  schemes}, BIT \textbf{53} (2013), no.~2, 497--525. \MR{3123856}

\bibitem{KLL18}
\bysame, \emph{On the discretisation in time of the stochastic {A}llen--{C}ahn
  equation}, Math. Nachr. \textbf{291} (2018), no.~5-6, 966--995. \MR{3795566}

\bibitem{LQ18}
Z.~Liu and Z.~Qiao, \emph{Strong approximation of stochastic {A}llen--{C}ahn
  equation with white noise}, IMA J. Numer. Anal.
  https://doi.org/10.1093/imanum/dry088 (2019).

\bibitem{MP17}
A.~K. Majee and A.~Prohl, \emph{Optimal strong rates of convergence for a
  space-time discretization of the stochastic {A}llen--{C}ahn equation with
  multiplicative noise}, Comput. Methods Appl. Math. \textbf{18} (2018), no.~2,
  297--311. \MR{3776047}

\bibitem{QW18}
R.~Qi and X.~Wang, \emph{Optimal error estimates of {G}alerkin finite element
  methods for stochastic {A}llen--{C}ahn equation with additive noise}, J Sci
  Comput. https://doi.org/10.1007/s10915-019-00973-8 (2019).

\bibitem{Ste56}
E.~M. Stein, \emph{Interpolation of linear operators}, Trans. Amer. Math. Soc.
  \textbf{83} (1956), 482--492. \MR{0082586}

\bibitem{Tho06}
V.~Thom\'ee, \emph{Galerkin finite element methods for parabolic problems},
  second ed., Springer Series in Computational Mathematics, vol.~25,
  Springer-Verlag, Berlin, 2006. \MR{1479170}

\bibitem{TW83}
V.~Thom\'{e}e and L.~B. Wahlbin, \emph{Maximum-norm stability and error
  estimates in {G}alerkin methods for parabolic equations in one space
  variable}, Numer. Math. \textbf{41} (1983), no.~3, 345--371. \MR{712117}

\bibitem{VVW08}
J.~van Neerven, M.~C. Veraar, and L.~Weis, \emph{Stochastic evolution equations
  in {UMD} {B}anach spaces}, J. Funct. Anal. \textbf{255} (2008), no.~4,
  940--993. \MR{2433958}

\bibitem{Wal05}
J.~B. Walsh, \emph{Finite element methods for parabolic stochastic {PDE}'s},
  Potential Anal. \textbf{23} (2005), no.~1, 1--43. \MR{2136207}

\bibitem{Wang16}
X.~Wang, \emph{Weak error estimates of the exponential {E}uler scheme for
  semi-linear {SPDE}s without {M}alliavin calculus}, Discrete Contin. Dyn.
  Syst. \textbf{36} (2016), no.~1, 481--497. \MR{3369232}

\bibitem{Wang18}
\bysame, \emph{An efficient explicit full discrete scheme for strong
  approximation of stochastic {A}llen--{C}ahn equation}, arXiv:1802.09413
  (2018).

\end{thebibliography}
\end{document}